[1994/06/01]

\documentclass[12pt,reqno,intlimits,english,centertags]{amsart}

\usepackage[T1]{fontenc}

\usepackage{ifthen,calc}
\usepackage{babel}
\usepackage{xspace}
\usepackage{amsmath}

\usepackage{lmodern}
\usepackage{amsthm}

\numberwithin{equation}{section}
\theoremstyle{plain}
\newtheorem{theorem}{Theorem}[section]
\newtheorem{lemma}[theorem]{Lemma}

\newtheorem{proposition}[theorem]{Proposition}
\theoremstyle{remark}
\newtheorem{remark}[theorem]{Remark}
\theoremstyle{definition}
\newtheorem{definition}[theorem]{Definition}

\DeclareMathOperator{\Div}{div}
\newcommand{\abs}[1]{\lvert#1\rvert}

\newcommand{\SpDim}{N}
\newcommand{\numberspacefont}{\boldsymbol}
\newcommand{\R}{\numberspacefont{R}}
\newcommand{\RN}{\R^{\SpDim}}

\newcommand{\Pposbase}[3][*]{\ifthenelse{\equal{#1}{*}}%
{(#2)_{#3}}{\left(#2\right)_{#3}}}

\newcommand{\di}{\,\textup{\textmd{d}}}
\newcommand{\eps}{\varepsilon}

\newcommand{\pder}[2]{\frac{\partial #1}{\partial #2}}

\newcommand{\grad}{\operatorname{\nabla}}
\DeclareMathOperator{\supp}{supp}
\newcommand{\der}[2]{\frac{\di #1}{\di #2}}

\newcommand{\ew}{f}
\newcommand{\exw}{g}

\newcommand{\phj}{\varphi}

\newcommand{\auxf}{\phj}

\newcommand{\excf}{G}

\newcommand{\rpol}{r}
\newcommand{\rpolf}{\hat\rpol}
\newcommand{\rw}{\rho}
\newcommand{\flopr}{\hat s}
\newcommand{\nonl}{F}
\newcommand{\rpolfb}{\tilde\rpol}
\newcommand{\nonlb}{h}
\newcommand{\Hf}{\tilde\varphi}

\newcommand{\ups}{\widetilde u}
\newcommand{\rat}{J}
\newcommand{\tar}{E}
\newcommand{\brt}{A}
\newcommand{\trb}{B}
\newcommand{\ram}{I}

\begin{document}
\title
[Equations with weights]{A supersolution approach to doubly degenerate
  parabolic equations with weights}%
\author{Daniele Andreucci}
\address{Department of Basic and Applied Sciences for Engineering\\Sapienza University of Rome\\via A. Scarpa 16 00161 Rome, Italy}
\email{daniele.andreucci@uniroma1.it}
\thanks{The first author is member of the Gruppo Nazionale
  per la Fisica Matematica (GNFM) of the Istituto Nazionale di Alta Matematica
  (INdAM). The first author thanks the PRIN 2022 project
``Mathematical Modelling of Heterogeneous Systems (MMHS)'',
financed by the European Union - Next Generation EU,
CUP B53D23009360006, Project Code 2022MKB7MM, PNRR M4.C2.1.1.}
\author{Anatoli F. Tedeev}
\address{Southern Mathematical Institute of VSC RAS\\53 Vatutina St. Vladikavkaz 362025, Russian Federation}
\email{a\_tedeev@yahoo.com}
\thanks{The second author carried out his work at the North-Caucasus Centre of Mathematical Research of the Vladikavkaz Scientific Centre of the Russian Academy of Sciences, agreement 075-02-2025-1633.}
\thanks{Keywords: Doubly degenerate parabolic equation, exponentially growing weights, weighted Sobolev inequality, finite speed of propagation, time decay estimates.\\AMS Subject classification: 35K55, 35K65, 35B40.}

\date{2025-12-23}

\begin{abstract}
  We consider the Cauchy problem in the Euclidean space for a doubly
  degenerate parabolic equation with a space-dependent exponential
  weight, where the exponent satisfies the doubling condition. In
  particular, both the so called logconvex and logconcave cases may be
  considered. Under the additional natural assumptions we construct
  supersolutions and subsolutions allowing us to control the precise
  sharp temporal decay bounds. We apply our results also to an
  equation with inhomogeneous density, via a suitable variable
  transformation.
\end{abstract}

\maketitle

\section{Introduction}
\label{s:intro}

We look at the Cauchy problem for the doubly degenerate weighted parabolic
equation
\begin{alignat}{2}
  \label{eq:pde}
  \ew(x)
  \pder{u}{t}
  -
  \Div\big(
  \ew(x)
  u^{m-1}
  \abs{\grad u}^{p-2}
  \grad u
  \big)
  &=
    0
    \,,
  &\qquad&
           \text{in $S_{T}$,}
  \\
  \label{eq:initd}
  u(x,0)
  &=
    u_{0}(x)
    \,,
  &\qquad&
           x\in\RN
           \,.
\end{alignat}
Here $S_{T}=\RN\times (0,T)$, $0<T\le +\infty$, $x=(x_{1},\dots,x_{N})$, $\grad u$ [respectively, $\Div$] is the gradient [respectively, the divergence] with respect to $x$, and we denote
\begin{equation}
  \label{eq:expweight}
  \ew(x)
  =
  e^{\exw(\abs{x})}
  \,,
  \qquad
  x\in\RN
  \,.
\end{equation}
We assume in this paper, even without further reference, that $u\ge 0$, $u_{0}\ge 0$, $u_{0}\in L^{\infty}(\RN)$ is compactly supported, and that
\begin{equation}
  \label{eq:degen}
  1<p<N
  \,,
  \qquad
  p+m-3>0
  \,.
\end{equation}
On $\exw$ we assume that $\exw\in C([0,+\infty])\cap C^{1}((0,+\infty))$ is such that $\exw(0)=0$, $\exw(s)>0$ for $s>0$, and for given $0<\alpha_{1}\le \alpha_{2}<p$
\begin{equation}
  \label{eq:powerlike}
  \alpha_{1}
  \frac{\exw(s)}{s}
  \le
  \exw'(s)
  \le
  \alpha_{2}
  \frac{\exw(s)}{s}
  \,,
  \qquad
  s>0
  \,.
\end{equation}

In \cite{Andreucci:Tedeev:2025} we obtained sup bounds for solutions
to \eqref{eq:pde}--\eqref{eq:initd} in the range $\alpha_{2}<1$, and
estimates of the support for all $\alpha_{2}<Np/(p-1)$. The approach
of \cite{Andreucci:Tedeev:2025} relies on suitable embedding
inequalities and iterative estimates, and is valid for radial or
compactly supported solutions; see also in this connection
\cite{Andreucci:Tedeev:2015} and \cite{Andreucci:Tedeev:2022b}. Here
we present instead new results both for the decay of the solutions in
the $L^{\infty}$ norm and for the finite speed of propagation, by
means of comparison with explicit supersolutions and subsolutions, in
the range $\alpha_{1}$, $\alpha_{2}\in (0,p)$. Our approach is
inspired by the paper \cite{Grillo:Muratori:Vazquez:2017}, where the
authors considered the porous media equation on manifolds. We treat
here the case of the doubly nonlinear equation, and, as a further
novelty even in the case of the porous media equation, we deal
with the case of non-power nonlinearities (satisfying
\eqref{eq:powerlike}). For example the Zygmund function
\begin{equation*}
  \exw(s)
  =
  s^{\alpha}
  [\log(s+c)]^{\beta}
  \,,
  \quad
  s\ge 0
  \,,
  \alpha>0
  \,,
  \beta>0
  \,,
  c>1
  \,,
\end{equation*}
falls in this class with $\alpha_{1}=\alpha$, $\alpha_{2}=\alpha+\beta$.
\\
We show here that the asymptotic rates found in \cite{Andreucci:Tedeev:2025} are indeed optimal. These bounds are given by
\begin{equation}
  \label{eq:sup_decay}
  u(x,t)
  \le
  \gamma
  \Big[
  \frac{\exw^{(-1)}(\log t)^{p}}{t\log t}
  \Big]^{\frac{1}{p+m-3}}
  \,,
  \quad
  x\in\RN
  \,,
  t\gg 1
  \,,
\end{equation}
and
\begin{equation}
  \label{eq:fsp}
  \bar R(t)
  \le
  \gamma
  \exw^{(-1)}(\log t)
  \,,
  \quad
  t\gg 1
  \,,
\end{equation}
where, for solutions whose support is bounded for $t\ge 0$ we define
\begin{equation}
  \label{eq:main_R}
  \bar R(t)
  =
  \inf\{
  \rpol>0
  \mid
  \supp u(t) \subset B_{\rpol}(0)
  \}
  \,.
\end{equation}
Here $\gamma>0$ is a constant depending on $u_{0}$.
According to our present results, the estimates in \eqref{eq:sup_decay}, \eqref{eq:fsp} are indeed sharp, at least for $\alpha_{2}<p$, and, even more, they are extended to such a range for bounded and compactly supported initial data. In addition the comparison results obtained here yield more precise pointwise estimates, tracking the dependence on $\abs{x}$, in comparison to the bounds for the $L^{\infty}(\RN)$ obtained in the quoted papers of the present authors. This is an obvious effect of the barrier function method, which however is in principle less general than the approach through integral inequalities. 

As in \cite{Grillo:Muratori:Vazquez:2017}, we may employ our barrier functions to investigate an equation with inhomogeneous density, i.e.,
\begin{equation}
  \label{eq:rad_N}
  \rw(\abs{x})
  \pder{v}{t}
  =
  \Div(
  v^{m-1}
  \abs{\grad v}^{p-2}
  \grad v
  )
  \,,
  \qquad
  x\in\R^{N}
  \,,
  t>0
  \,.
\end{equation}
For the concept of solutions of \eqref{eq:rad_N}, and other related
information, we refer the reader to \cite{Andreucci:Tedeev:2021b}. In that paper, in fact in a more general setting, we considered equations modeled after \eqref{eq:rad_N}, with, roughly speaking, $\rw(s)$ behaving like $s^{-\alpha}$, $\alpha<p$ for large $s>0$. Here we deal with the new case \eqref{eq:rad_rw_est}, in which essentially $\rw(s)$ behaves like the power $s^{-p}$ multiplied by a powerlike function of $\log s$; this factor may be increasing or decreasing in $s$ according to the values of $\alpha_{1}$, $\alpha_{2}$. We refer to Remark~\ref{r:rad_N} for further comments on this point.

The Cauchy problem for nonlinear, and linear, parabolic equations
involving coefficients strongly depending on the space variable has
received great attention in the literature. We quote
\cite{Grigoryan:2006} which analyzes stochastical completeness and
other qualitative properties of weighted manifolds. In
\cite{Grillo:Muratori:2016} the authors study the smoothness and
temporal decay estimates of the solution to the Cauchy problem for the
porous media equation on Cartan-Hadamard manifolds supporting
Poincar\'{e} inequality. See also \cite{Grillo:Muratori:Vazquez:2017}
for precise two-sides space-time estimates of the solution to the  same problem,
in the setting of
various classes of Cartan-Hadamard manifolds, 
\cite{Muratori:2021} for a survey of results, and again
\cite{Muratori:Roncoroni:2022}
for new Sobolev type inequalities as well as sharp bounds of solutions.
We also quote \cite{Tedeev:2007}, dealing with the influence of an inhomogeneous density on
the interface blow up phenomenon, on finite speed of
propagation, and in general on the 
behavior of  the solution to the Cauchy problem for doubly degenerate
parabolic equations. 
Finally
\cite{Tedeev:2025} investigates the blow up of the solution itself in the same setting.

\subsection{Main results}
\label{s:main}

We define here the spaces $L^{p}_{\ew}(\RN)$ in a standard way, as the $L^{p}$ spaces relative to the measure $\ew(x)\di x$.

\begin{definition}
  \label{d:weaksol}
  A function $u\ge 0$ is a weak solution to \eqref{eq:pde} if
  \begin{equation*}
    u\in C([0,+\infty); L^{1}_{\ew}(\RN)) \cap L^{\infty}(\RN\times(\bar t,+\infty))
  \end{equation*}
  for all $\bar t>0$, and
  \begin{equation*}
    \grad u^{\frac{p+m-2}{p-1}}
    \in
    L^{p}_{\textup{loc}}
    (S_{\infty})
    \,.
  \end{equation*}
  In addition we require the standard integral formulation, i.e., for all $\eta\in C_{0}^{\infty}(S_{\infty})$,
  \begin{equation}
    \label{eq:weaksol_a}
    \int_{0}^{+\infty}
    \int_{\RN}
    \{
    -
    u
    \eta_{t}
    +
    u^{m-1}
    \abs{\grad u}^{p-2}
    \grad u
    \grad \eta
    \}
    \ew(x)
    \di x
    \di t
    =
    0
    \,.
  \end{equation}
  \\
  If we also have $u_{t}\in L^{\infty}_{\textup{loc}}(0,+\infty); L^{1}_{f\,\textup{loc}}(\RN))$, then $u$ is a strong solution.
  \\
  Supersolutions [respectively, subsolutions] to \eqref{eq:pde} are defined similarly, with the difference that, for $\eta\ge0$, we require in \eqref{eq:weaksol_a} an inequality sign $\ge $ [respectively, $\le$].
  \\
  Solutions to the Cauchy problem \eqref{eq:pde}--\eqref{eq:initd} are solutions to \eqref{eq:pde} satisfying $u(0)=u_{0}$ in the sense of $C([0,+\infty); L^{1}_{\ew}(\RN))$.
\end{definition}

Existence and comparison results for strong solutions to the problems which we consider here can be proved, at least for sufficiently regular initial data, according for example to the methods of \cite{Degtyarev:Tedeev:2012}, \cite{Tsutsumi:1988}; here we focus rather on estimates of such solutions.

When we consider radial solutions of the form $u(x,t)=U(\rpol,t)$,
$\rpol=\abs{x}$, the equation \eqref{eq:pde} can be written as
\begin{equation}
  \label{eq:pde_rad}
  \rpol^{N-1}
  \ew(\rpol)
  \pder{U}{t}
  =
  \pder{}{\rpol}
  \big[
  \rpol^{N-1}
  \ew(\rpol)
  U^{m-1}
  \abs{U_{\rpol}}^{p-2}
  U_{\rpol}
  \big]
  \,,
  \qquad
  \rpol>0
  \,,
  t>0
  \,,
\end{equation}
which in turn can be rephrased as
\begin{equation}
  \label{eq:sss_rad}
  \begin{split}
    \pder{U}{t}
    &=
      \pder{}{\rpol}
      \big[
      U^{m-1}
      \abs{U_{\rpol}}^{p-2}
      U_{\rpol}
      \big]
    \\
    &\quad+
      \Big(
      \frac{N-1}{\rpol}
      +
      \exw'
      \Big)
      \big[
      U^{m-1}
      \abs{U_{\rpol}}^{p-2}
      U_{\rpol}
      \big]
      \,,
      \qquad
      \rpol>0
      \,,
      t>0
      \,.
  \end{split}
\end{equation}
Of course the differential inequalities characterizing radial supersolutions and subsolutions can be rephrased similarly.
Specifically, we consider supersolutions and subsolutions of the form
\begin{alignat}2
  \label{eq:sss_sup}
  \ups(x,t)
  &=
    \frac{C_{*}}{(t+t_{0})^{\frac{1}{p+m-3}}}
    \Big[
    \tar(\tau)^{\frac{1}{p-1}}
    -
    \rat(\rpol)^{\frac{1}{p-1}}
    \Big]^{\frac{p-1}{p+m-3}}_{+}
    \,,
    &\quad&
            \rpol\ge \rpol_{0}
            \,,
  \\
  \label{eq:sss_ups_2}
  \ups(x,t)
  &=
    \frac{C_{*}}{(t+t_{0})^{\frac{1}{p+m-3}}}
    \Big[
    \tar(\tau)^{\frac{1}{p-1}}
    -
    \ram(\rpol)
    \Big]^{\frac{p-1}{p+m-3}}_{+}
    \,,
    &\quad&
            \rpol<\rpol_{0}
            \,,
\end{alignat}
where $\rpol=\abs{x}$, and
\begin{gather*}
  \excf(\rpol)
  =
  \int_{0}^{\rpol}
  \frac{\exw(s)}{s}
  \di s
  \,,
  \quad
  \rpol\ge 0
  \,;
  \qquad
  \rat(\rpol)
  =
  \frac{\rpol^{p}}{\excf(\rpol)}
  \,,
  \quad
  \rpol>0
  \,.
  \\
  \tau
  =
  \varGamma
  \log(t+t_{0})
  \,,
  \quad
  t\ge 0
  \,;
  \\
  \tar(\tau)
  =
  \rat(\excf^{(-1)}(\tau))
  =
  \frac{\excf^{(-1)}(\tau)^{p}}{\tau}
  \,,
  \quad
  \tau>0
  \,;
  \\
  \ram(\rpol)
  =
  \nu_{0}
  \Big(
  \frac{\rpol^{p}}{\excf(\rpol_{0})}
  \Big)^{\frac{1}{p-1}}
  +
  (1-\nu_{0})
  \rat(\rpol_{0})^{\frac{1}{p-1}}
  \,,
  \quad
  \rpol>0
  \,.
\end{gather*}
Note that the integral defining $\excf$ converges according to \eqref{eq:powerlike_lt}.

Here $C_{*}$, $\varGamma$, $t_{0}> 1$, $\rpol_0$ and $\nu_{0}$ are positive constants to be chosen. In particular, $\nu_{0}\in(0,1)$ will be selected so that $\ups$ and $\ups_{\rpol}$ are continuous even at $\rpol=\rpol_{0}$; in fact, clearly, the continuity of $\ups$ holds true for any $\nu_{0}\in(0,1)$.

We also remark that the radius $\tilde \rpol(t)$ of the support of $\ups(t)$ is characterized by $\tar(\tau)=\rat(\tilde\rpol(t))$, i.e., by $\tilde\rpol(t)=\excf^{(-1)}(\tau)$; according to Lemma~\ref{l:sss_excf} and \eqref{eq:sss_powerlike_gt_excf} this yields
\begin{equation}
  \label{eq:sup_sub_radius}
  \gamma_{1}
  \exw^{(-1)}(\log t)
  \le
  \tilde \rpol(t)
  \le
  \gamma_{2}
  \exw^{(-1)}(\log t)
  \,,
  \qquad
  t>t_{0}
  \,,
\end{equation}
for two suitable constants $\gamma_{1}$, $\gamma_{2}>0$.

\begin{proposition}
  \label{p:main_sup_sub}
  The constants $C_{*}$, $\varGamma$, $t_{0}> 1$, $\rpol_0$ and
  $\nu_{0}$ can be selected so that $\ups$ as defined in
  \eqref{eq:sss_sup}--\eqref{eq:sss_ups_2} is a strong supersolution to
  \eqref{eq:pde}. Alternatively, we may select such constants so that
  $\ups$ is a strong subsolution to \eqref{eq:pde}.
\end{proposition}

\begin{theorem}
  \label{t:sup_sup}
  Assume \eqref{eq:degen} and \eqref{eq:powerlike}. Let $u$ be a strong solution to \eqref{eq:pde}--\eqref{eq:initd} with a compactly supported and bounded data $u_{0}$.
  \\
  Then there exist constants $t_{0}$, $C_{*}$, $\rpol_{0}$ and $\varGamma$ such that $\ups$ as in \eqref{eq:sss_sup}--\eqref{eq:sss_ups_2} is a supersolution and
  \begin{equation}
    \label{eq:sup_sup_n}
    u(x,t)
    \le
    \ups(x,t)
    \,,
    \qquad
    (x,t)\in S_{\infty}
    \,.
  \end{equation}
  We have also
  \begin{equation}
    \label{eq:sup_sup_nn}
    \supp u(t)
    \subset
    B_{R(t)}(0)
    \,,
    \quad
    R(t)
    =
    \gamma
    \exw^{(-1)}(\log t)
    \,,
    \qquad
    t>t_{0}
    \,,
  \end{equation}
  for a suitable constant $\gamma>0$. The constants here depend on the parameters appearing in the assumptions and on $\sup_{\RN} u_{0}$, $R(0)$.
\end{theorem}

\begin{theorem}
  \label{t:sub_sub}
  Assume \eqref{eq:degen} and \eqref{eq:powerlike}. Let $u$ be a strong solution to \eqref{eq:pde}--\eqref{eq:initd} with a compactly supported and bounded data $u_{0}$, such that $u_{0}(x)\ge \eps>0$ for $\abs{x}\le\ell$, $\ell>0$.
  \\
  Then there exist constants $t_{0}$, $C_{*}$, $\rpol_{0}$ and $\varGamma$ such that $\ups$ as in \eqref{eq:sss_sup}--\eqref{eq:sss_ups_2} is a subsolution and
  \begin{equation}
    \label{eq:sub_sub_n}
    u(x,t)
    \ge
    \ups(x,t)
    \,,
    \qquad
    (x,t)\in S_{\infty}
    \,.
  \end{equation}
  We have also
  \begin{equation}
    \label{eq:sub_sub_nn}
    \supp u(t)
    \supset
    B_{R(t)}(0)
    \,,
    \quad
    R(t)
    =
    \gamma
    \exw^{(-1)}(\log t)
    \,,
    \qquad
    t>t_{0}
    \,,
  \end{equation}
  for a suitable constant $\gamma>0$. The constants here depend on the parameters appearing in the assumptions and on $\eps$, $\ell$.
\end{theorem}

\begin{proof}[Proofs of Theorems \ref{t:sup_sup} and \ref{t:sub_sub}]
  The proof follows straightforwardly from an application of the method given in the proof of \cite[Theorem~1]{Tsutsumi:1988}, when we take into account Proposition~\ref{p:main_sup_sub} and Lemmas \ref{l:sss_comp} and \ref{l:ssb_comp}.
\end{proof}

Next we apply our results to a different kind of equation, that is \eqref{eq:rad_N}; see also \cite{Vazquez:2015}, \cite{Grillo:Muratori:Vazquez:2017} for the case of the porous media equation.

\begin{theorem}
  \label{t:transform}
Radial solutions $U(\rpol,t)$ [respectively, supersolutions and subsolutions] of \eqref{eq:pde} correspond, thru a suitable transformation of the space variable, to radial solutions $v(s,t)$ [respectively, supersolutions and subsolutions] of \eqref{eq:rad_N},
for a suitable positive function $\rw\in C([0,+\infty))$ such that $\rw(0)=1$ and
\begin{equation}
  \label{eq:rad_rw_est}
  \rw(s)
  \sim
  \frac{\exw^{(-1)}(\log s)^{p}}{(\log s)^{p}}
  \frac{1}{s^{p}}
  \,.
\end{equation}
In fact, the transformation $r=\rpolf(s)$, $s>0$, is such that
\begin{equation}
  \label{eq:rad_r_asy_n}
  \rpolf(s)
  \sim
  \exw^{(-1)}(\log s)
  \,.
\end{equation}
\end{theorem}

Here and in the following we use the notation, for positive functions $F_{1}$, $F_{2}$,
\begin{equation}
  \label{eq:rad_sim}
  F_{1}(s)\sim F_{2}(s)
  \quad
  \text{if and only if}
  \quad
  C_{1}
  \le
  \frac{F_{2}(s)}{F_{1}(s)}
  \le
  C_{2}
  \,,
  \quad
  s\ge \bar s
  \,,
\end{equation}
for suitable $C_{1}$, $C_{2}$, $\bar s>0$.

Owing to Theorem~\ref{t:transform}, supersolutions and subsolutions of the type \eqref{eq:sss_sup}--\eqref{eq:sss_ups_2} yield radial counterparts for equation \eqref{eq:rad_N}. Thus, from our  Theorems \ref{t:sup_sup}, \ref{t:sub_sub} the following comparison results follow immediately.

\begin{theorem}
  \label{t:sup_sup_dens}
  Let $v$ be a strong solution to \eqref{eq:rad_N}, for the specific $\rw$ found in Theorem~\ref{t:transform}, under assumptions \eqref{eq:degen} and \eqref{eq:powerlike}, such that $v(x,0)$ is compactly supported and bounded.
  \\
  Then there exist constants $t_{0}$, $C_{*}$, $\rpol_{0}$ and $\varGamma$ such that $\ups(\rpolf(\abs{x}),t)$ as in \eqref{eq:sss_sup}--\eqref{eq:sss_ups_2} is a supersolution to \eqref{eq:rad_N}, and
  \begin{equation}
    \label{eq:sup_sup_dens_n}
    v(x,t)
    \le
    \ups(\rpolf(\abs{x}),t)
    \,,
    \qquad
    (x,t)\in S_{\infty}
    \,.
  \end{equation}
  We have also
  \begin{equation}
    \label{eq:sup_sup_dens_nn}
    \supp v(t)
    \subset
    B_{\tilde R(t)}(0)
    \,,
    \qquad
    \log \tilde R(t)
    \sim
    \log t
    \,,
    \qquad
    t>t_{0}
    \,.
  \end{equation}
  The constants involved in \eqref{eq:sup_sup_dens_n}--\eqref{eq:sup_sup_dens_nn} depend on the parameters appearing in the assumptions and on $v(x,0)$.
\end{theorem}

\begin{theorem}
  \label{t:sub_sub_dens}
  Let $v$ be a strong solution to \eqref{eq:rad_N}, for the specific $\rw$ found in Theorem~\ref{t:transform}, under assumptions \eqref{eq:degen} and \eqref{eq:powerlike}, such that $v(x,0)$ is compactly supported, bounded and
$v(x,0)\ge \eps>0$ for $\abs{x}\le\ell$, $\ell>0$.
  \\
  Then there exist constants $t_{0}$, $C_{*}$, $\rpol_{0}$ and $\varGamma$ such that $\ups(\rpolf(\abs{x}),t)$ as in \eqref{eq:sss_sup}--\eqref{eq:sss_ups_2} is a subsolution \eqref{eq:rad_N}, and
  \begin{equation}
    \label{eq:sub_sub_dens_n}
    v(x,t)
    \ge
    \ups(\rpolf(\abs{x}),t)
    \,,
    \qquad
    (x,t)\in S_{\infty}
    \,.
  \end{equation}
  We have also
  \begin{equation}
    \label{eq:sub_sub_dens_nn}
    \supp v(t)
    \supset
    B_{\tilde R(t)}(0)
    \,,
    \quad
    \log\tilde R(t)
    \sim
    \log t
    \,,
    \qquad
    t>t_{0}
    \,,
  \end{equation}
  The constants involved in \eqref{eq:sub_sub_dens_n}--\eqref{eq:sub_sub_dens_nn} depend on the parameters appearing in the assumptions and on $v(x,0)$.
\end{theorem}

\begin{remark}
  \label{r:rad_N}
  Clearly the sup estimate valid for $v$ as in Theorems
  \ref{t:sup_sup_dens} and \ref{t:sub_sub_dens} is the same as in
  \eqref{eq:sup_decay}, which should be compared with the result valid for $\rw(s)\sim s^{-\alpha}$, $\alpha>p$, that is the universal bound $t^{-1/(p+m-3)}$ (see \cite{Andreucci:Tedeev:2021b}). Our present bound however depends on the initial data, as the known bound valid in the subcritical case $\alpha<p$, whose decay rate is $t^{-(N-\alpha)/K}$, $K=(N-\alpha)(p+m-3)+p-\alpha$.
\end{remark}

\textbf{Plan of the paper.}
Section~\ref{s:pre} contains some elementary auxiliary results.
Then Proposition~\ref{p:main_sup_sub} is proved in Sections \ref{s:sss} (supersolutions) and \ref{s:ssb} (subsolutions).
Finally in Section~\ref{s:rad} we carry out the transformation of variables linking \eqref{eq:pde} and \eqref{eq:rad_N}. 

\section{Technical preliminaries}
\label{s:pre}

Note that as a consequence of \eqref{eq:powerlike} we have
\begin{equation}
  \label{eq:powerlike_inv}
  \alpha_{2}^{-1}
  \frac{\exw^{(-1)}(t)}{t}
  \le
  \der{\exw^{(-1)}}{t}(t)
  \le
  \alpha_{1}^{-1}
  \frac{\exw^{(-1)}(t)}{t}
  \,,
  \qquad
  t>0
  \,,
\end{equation}
and also the standard inequalities
\begin{gather}
  \label{eq:powerlike_gt}
  \lambda^{\alpha_{1}}
  \exw(\rpol)
  \le
  \exw(\lambda\rpol)
  \le
  \lambda^{\alpha_{2}}
  \exw(\rpol)
  \,,
  \qquad
  \lambda\ge 1
  \,,
  \\
  \label{eq:powerlike_lt}
  \lambda^{\alpha_{2}}
  \exw(\rpol)
  \le
  \exw(\lambda\rpol)
  \le
  \lambda^{\alpha_{1}}
  \exw(\rpol)
  \,,
  \qquad
  \lambda\le 1
  \,.
\end{gather}

We need the following technical Lemmas. As a matter of fact, in Lemmas \ref{l:sss_excf}, \ref{l:sss_excf2} and \ref{l:sss_rat} we only need $0<\alpha_{1}\le \alpha_{2}$; in Lemma~\ref{l:sss_comb} we also assume $\alpha_{2}<p$.

\begin{lemma}
  \label{l:sss_excf}
  We have
  \begin{equation}
    \label{eq:sss_excf_n}
    \frac{\exw(\rpol)}{\alpha_{2}}
    \le
    \excf(\rpol)
    \le
    \frac{\exw(\rpol)}{\alpha_{1}}
    \,,
    \qquad
    \rpol>0
    \,,
  \end{equation}
  and its obvious consequence
  \begin{equation}
    \label{eq:sss_excf_nn}
    \alpha_{1}
    \frac{\excf(\rpol)}{\rpol}
    \le
    \excf'(\rpol)
    =
    \frac{\exw(\rpol)}{\rpol}
    \le
    \alpha_{2}
    \frac{\excf(\rpol)}{\rpol}
    \,,
    \qquad
    \rpol>0
    \,.
  \end{equation}
\end{lemma}

\begin{proof}
  We get from \eqref{eq:powerlike_lt}
  \begin{equation*}
    \excf(\rpol)
    =
    \int_{0}^{1}
    \frac{\exw(\lambda \rpol)}{\lambda}
    \di\lambda
    \le
    \exw(\rpol)
    \int_{0}^{1}
    \lambda^{\alpha_{1}-1}
    \di\lambda
    =
    \frac{\exw(\rpol)}{\alpha_{1}}
    \,,
  \end{equation*}
  i.e., the second inequality in \eqref{eq:sss_excf_n}. The first inequality can be proved in the same way.
\end{proof}

By virtue of \eqref{eq:sss_excf_nn}, it follows that $\excf$ and $\excf^{(-1)}$ satisfy estimates similar to \eqref{eq:powerlike_inv}--\eqref{eq:powerlike_lt}. We note specifically that for $\rpol>0$
\begin{equation}
  \label{eq:sss_powerlike_gt_excf}
  \lambda^{\frac{1}{\alpha_{2}}}
  \excf^{(-1)}(\rpol)
  \le
  \excf^{(-1)}(\lambda\rpol)
  \le
  \lambda^{\frac{1}{\alpha_{1}}}
  \excf^{(-1)}(\rpol)
  \,,
  \qquad
  \lambda\ge 1
  \,.
\end{equation}

\begin{lemma}
  \label{l:sss_excf2}
  The following estimates are in force:
  \begin{equation}
    \label{eq:sss_excf2_n}
    \eta_{1}
    \excf(\rpol)
    \le
    \rpol^{2}
    \excf''(\rpol)
    \le
    \eta_{2}
    \excf(\rpol)
    \,,
    \qquad
    \rpol>0
    \,,
  \end{equation}
  where we may select
  \begin{gather*}
    \eta_{1}
    =
    (\alpha_{1}-1)
    \alpha_{2}
    \,,
    \quad
    \alpha_{1}<1
    \,;
    \qquad
    \eta_{1}
    =
    (\alpha_{1}-1)
    \alpha_{1}
    \,,
    \quad
    \alpha_{1}\ge 1
    \,;
    \\
    \eta_{2}
    =
    (\alpha_{2}-1)
    \alpha_{1}
    \,,
    \quad
    \alpha_{2}<1
    \,;
    \qquad
    \eta_{2}
    =
    (\alpha_{2}-1)
    \alpha_{2}
    \,,
    \quad
    \alpha_{2}\ge 1
    \,.
  \end{gather*}
\end{lemma}

\begin{proof}
  We have in any case
  \begin{equation*}
    (\alpha_{2}-1)\exw(\rpol)
    \ge
    \rpol^{2}\excf''(\rpol)
    =
    \rpol\exw'(\rpol)-\exw(\rpol)
    \ge
    (\alpha_{1}-1)\exw(\rpol)
    \,.
  \end{equation*}
  Then the claim follows from using \eqref{eq:sss_excf_n} according to the various possible cases.
\end{proof}

\begin{lemma}
  \label{l:sss_rat}
  We have
  \begin{equation}
    \label{eq:sss_rat_n}
    (p-\alpha_{2})
    \frac{\rat(\rpol)}{\rpol}
    \le
    \rat'(\rpol)
    \le
    (p-\alpha_{1})
    \frac{\rat(\rpol)}{\rpol}
    \,,
    \qquad
    \rpol>0
    \,,
  \end{equation}
  and
  \begin{equation}
    \label{eq:sss_rat_nn}
    c_{1}
    \frac{\rat(\rpol)}{\rpol^{2}}
    \le
    \rat''(\rpol)
    \le
    c_{2}
    \frac{\rat(\rpol)}{\rpol^{2}}
    \,,
    \qquad
    \rpol>0
    \,,
  \end{equation}
  where
  \begin{alignat*}2
    c_{1}
    &=
      p
      (p-1-2\alpha_{2})
      +
      2\alpha_{1}^{2}
      -
      \alpha_{2}
      (\alpha_{2}-1)
      \,,
    &\qquad&
             \alpha_{2}\ge 1
             \,,
    \\
    c_{1}
    &=
      p
      (p-1-2\alpha_{2})
      +
      2\alpha_{1}^{2}
      -
      \alpha_{1}
      (\alpha_{2}-1)
      \,,
    &\qquad&
             \alpha_{2}< 1
             \,,
    \\
    c_{2}
    &=
      p
      (p-1-2\alpha_{1})
      +
      2\alpha_{2}^{2}
      -
      \alpha_{1}
      (\alpha_{1}
      -
      1
      )
      \,,
    &\qquad&
             \alpha_{1}\ge 1
             \,,
    \\
    c_{2}
    &=
      p
      (p-1-2\alpha_{1})
      +
      2\alpha_{2}^{2}
      -
      \alpha_{2}
      (\alpha_{1}
      -
      1
      )
      \,,
    &\qquad&
             \alpha_{1}< 1
             \,.
  \end{alignat*}
\end{lemma}

\begin{proof}
  The inequalities in \eqref{eq:sss_rat_n}, when we also invoke \eqref{eq:sss_excf_nn}, follow from
  \begin{equation}
    \label{eq:sss_rat_j}
    \rat'(\rpol)
    =
    \frac{\rat(\rpol)}{\rpol}
    \Big[
    p
    -
    \frac{r\excf'(\rpol)}{\excf(\rpol)}
    \Big]
    \,.
  \end{equation}

  Then we calculate
  \begin{equation}
    \label{eq:sss_rat_i}
    \rat''(r)
    =
    \frac{\rat(\rpol)}{\rpol^{2}}
    \Big[
    p(p-1)
    -
    2p\rpol
    \frac{\excf'(\rpol)}{\excf(\rpol)}
    -
    \rpol^{2}
    \frac{\excf''(\rpol)}{\excf(\rpol)}
    +
    2\rpol^{2}
    \frac{\excf'(\rpol)^{2}}{\excf(\rpol)^{2}}
    \Big]
    \,.
  \end{equation}
  The proof is then concluded by estimating
  the ratios $\excf'/\excf$ in \eqref{eq:sss_rat_i} by means of \eqref{eq:sss_excf_nn}, and the ratio $\excf''/\excf$  by means of \eqref{eq:sss_excf2_n}. 
\end{proof}

We do not make any claim on the signs of the constants $c_{i}$ appearing in Lemma~\ref{l:sss_excf2}.

\begin{lemma}
  \label{l:sss_comb}
  The function
  \begin{equation*}
    \auxf(\rpol)
    =
    \rat(\rpol)^{\frac{1}{p-1}}
    \Big(
    \delta_{1}
    +
    \frac{\delta_{2}}{\excf(\rpol)}
    \Big)
    \,,
  \end{equation*}
  where $\delta_{1}>0$, $\delta_{2}>0$ are given  constants, is increasing in the interval $(\rpol_{1},+\infty)$, for a suitable $\rpol_{1}$ given by 
  \begin{equation}
    \label{eq:sss_comb_n}
    \excf(\rpol_{1})
    =
    \frac{\delta_{2}}{\delta_{1}}
    \frac{p(\alpha_{2}-1)_{+}}{p-\alpha_{2}}
    \,.
  \end{equation}
\end{lemma}

\begin{proof}
  By direct differentiation, and with the help of Lemmas \ref{l:sss_excf} and \ref{l:sss_rat}, we get, on using \eqref{eq:sss_rat_j} in the third equality and \eqref{eq:sss_rat_n} in the inequality,
  \begin{equation}
    \label{eq:sss_comb_i}
    \begin{split}
      \auxf'(\rpol)
      &=
        \rat(\rpol)^{\frac{1}{p-1}-1}
        \frac{\rat'(\rpol)}{p-1}
        \Big(
        \delta_{1}
        +
        \frac{\delta_{2}}{\excf(\rpol)}
        \Big)
        -
        \delta_{2}
        \rat(\rpol)^{\frac{1}{p-1}}
        \frac{\excf'(\rpol)}{\excf(\rpol)^{2}}
        \\
      &=
        \delta_{1}
        \rat(\rpol)^{\frac{1}{p-1}-1}
        \frac{\rat'(\rpol)}{p-1}
        +
        \delta_{2}
        \frac{\rat(\rpol)^{\frac{1}{p-1}}}{\rpol\excf(\rpol)}
        \Big[
        \frac{\rpol}{p-1}
        \frac{\rat'(\rpol)}{\rat(\rpol)}
        -
        \frac{\excf'(\rpol)\rpol}{\excf(\rpol)}
        \Big]
        \\
      &=
        \delta_{1}
        \rat(\rpol)^{\frac{1}{p-1}-1}
        \frac{\rat'(\rpol)}{p-1}
        +
        \frac{p}{p-1}
        \delta_{2}
        \frac{\rat(\rpol)^{\frac{1}{p-1}}}{\rpol\excf(\rpol)}
        \Big[
        1
        -
        \frac{\excf'(\rpol)\rpol}{\excf(\rpol)}
        \Big]
        \\
      &\ge
        \frac{\rat(\rpol)^{\frac{1}{p-1}}}{\rpol}
        \Big\{
        \delta_{1}
        \frac{p-\alpha_{2}}{p-1}
        +
        \frac{p}{p-1}
        \frac{\delta_{2}}{\excf(\rpol)}
        \Big[
        1
        -
        \frac{\excf'(\rpol)\rpol}{\excf(\rpol)}
        \Big]
        \Big\}
        \,.
    \end{split}
  \end{equation}
  The claim follows from the inequality $\auxf'(\rpol)\ge 0$, which can be proved, for $\rpol$ as in \eqref{eq:sss_comb_n}, on applying next \eqref{eq:sss_excf_n}.
\end{proof}

\section{Self similar supersolutions}
\label{s:sss}

We assume here $\varGamma\ge 1$ and
\begin{equation}
  \label{eq:sss_ups_t0r0}
  \excf(\rpol_{0})
  <
  \log t_{0}
  \le
  \varGamma
  \log t_{0}
  \,,
\end{equation}
so that the support of $\ups(t)$ contains the ball $\rpol\le\rpol_{0}$ for all $t\ge 0$; this follows from the fact that $\rat$, and therefore $\tar$, is an increasing function as proved in Lemma~\ref{l:sss_rat}. By the same token, the support of $\ups(t)$ is the ball of radius $\excf^{(-1)}(\tau)$.

For the reader's convenience we recall that we assume $\alpha_{2}<p$, and note that $\rpol_{0}$ will be chosen as in \eqref{eq:sss_ups_r0}, $C_{*}$ as in \eqref{eq:sss_ups_C}, $\nu_{0}$ as in \eqref{eq:sss_ups_c1}, $\varGamma$ as in \eqref{eq:sss_ups_Gamma0}, while $t_{0}$ is required to satisfy \eqref{eq:sss_ups_t0r0} and \eqref{eq:sss_ups_ineqs5}.

1) Case $\rpol\ge \rpol_{0}$.
With the notation above, we write
\begin{equation}
  \label{eq:sss_sup_2}
  \ups(x,t)
  =
  \frac{C_{*}\brt(\rpol,\tau)^{\frac{p-1}{p+m-3}}}{(t+t_{0})^{\frac{1}{p+m-3}}}
  \,,
  \quad
  \brt(\rpol,\tau)
  :=
  \big[
  \tar(\tau)^{\frac{1}{p-1}}
  -
  \rat(\rpol)^{\frac{1}{p-1}}
  \big]_{+}
  \,.
\end{equation}
Clearly we may restrict in our calculations to the open set where $\brt>0$. Thus we compute
\begin{equation}
  \label{eq:sss_ups_pow}
  \ups(x,t)^{m-1}
  =
  \frac{C_{*}^{m-1}\brt(\rpol,\tau)^{\frac{(m-1)(p-1)}{p+m-3}}}{(t+t_{0})^{\frac{m-1}{p+m-3}}}
  \,,
\end{equation}
and
\begin{equation}
  \label{eq:sss_ups_rder}
  \ups_{\rpol}(x,t)
  =
  -
  \frac{1}{p+m-3}
  \frac{C_{*}\brt(\rpol,\tau)^{\frac{p-1}{p+m-3}-1}}{(t+t_{0})^{\frac{1}{p+m-3}}}
  \rat(\rpol)^{\frac{1}{p-1}-1}
  \rat'(\rpol)
  \,.
\end{equation}
Thus, on recalling $\rat'(\rpol)\ge 0$, we obtain
\begin{equation}
  \label{eq:sss_ups_diff}
  \begin{split}
    I_{1}
    &:=
      \ups^{m-1}
      \abs{\ups_{\rpol}}^{p-2}
      \ups_{\rpol}
    \\
    &=
      -
      \frac{
      C_{*}^{p+m-2}
      }{
      (p+m-3)^{p-1}
      }
      \,
      \frac{
      \brt(\rpol,\tau)^{\frac{p-1}{p+m-3}}
      }{
      (t+t_{0})^{\frac{p+m-2}{p+m-3}}
      }
      \rat(\rpol)^{2-p}
      \rat'(\rpol)^{p-1}
      \,.
  \end{split}
\end{equation}
We have for the first term on the right hand side of \eqref{eq:sss_rad}
\begin{equation}
  \label{eq:sss_ups_diff2}
  \pder{I_{1}}{\rpol}
  =
  -
  \frac{
    C_{*}^{p+m-2}
  }{
    (p+m-3)^{p-1}
  }
  \,
  \frac{
    1
  }{
    (t+t_{0})^{\frac{p+m-2}{p+m-3}}
  }
  I_{2}
  \,,
\end{equation}
where
\begin{equation*}
  \begin{split}
    I_{2}
    &=
      \pder{}{\rpol}
      \big(
      \brt(\rpol,\tau)^{\frac{p-1}{p+m-3}}
      \big)
      \rat(\rpol)^{2-p}
      \rat'(\rpol)^{p-1}
    \\
    &\quad+
      \brt(\rpol,\tau)^{\frac{p-1}{p+m-3}}
      \pder{}{\rpol}
      \big(
      \rat(\rpol)^{2-p}
      \rat'(\rpol)^{p-1}
      \big)
      \,.
  \end{split}
\end{equation*}
We calculate by direct differentiation
\begin{equation}
  \label{eq:sss_ups_diff3}
  \begin{split}
    I_{2}
    &=
      \brt(\rpol,\tau)^{\frac{p-1}{p+m-3}-1}
      \Big\{
      -
      \frac{1}{p+m-3}
      \rat(\rpol)^{-\frac{p(p-2)}{p-1}}
      \rat'(\rpol)^{p}
    \\
    &\quad+
      \brt(\rpol,\tau)
      \big[
      (2-p)
      \rat(\rpol)^{1-p}
      \rat'(\rpol)^{p}
    \\
    &\qquad+
      (p-1)
      \rat(\rpol)^{2-p}
      \rat'(\rpol)^{p-2}
      \rat''(\rpol)
      \big]
      \Big\}
      =:
      \brt(\rpol,\tau)^{\frac{p-1}{p+m-3}-1}
      I_{3}
      \,.
  \end{split}
\end{equation}

Then for the time derivative we calculate
\begin{equation}
  \label{eq:sss_ups_tder}
  \begin{split}
    \ups_{t}(x,t)
    &=
      -
      \frac{1}{p+m-3}
      \frac{
      C_{*}\brt(\rpol,\tau)^{\frac{p-1}{p+m-3}}
      }{
      (t+t_{0})^{\frac{p+m-2}{p+m-3}}
      }
      \\
    &\quad
      +
      \frac{p-1}{p+m-3}
      \frac{
      C_{*}\brt(\rpol,\tau)^{\frac{2-m}{p+m-3}}
      }{
      (t+t_{0})^{\frac{1}{p+m-3}}
      }
      \der{}{t}
      \big(\tar(\tau)^{\frac{1}{p-1}}\big)
      \,.
  \end{split}
\end{equation}
Then we note that
\begin{equation}
  \label{eq:sss_ups_ttar}
  \begin{split}
    \der{}{t}
    \big(\tar(\tau)^{\frac{1}{p-1}}\big)
    &=
      \der{\tau}{t}
      \der{}{\tau}
      \big(\tar(\tau)^{\frac{1}{p-1}}\big)
    \\
    &=
      \frac{\varGamma}{t+t_{0}}
      \frac{1}{p-1}
      \tar(\tau)^{\frac{2-p}{p-1}}
      \frac{\excf^{(-1)}(\tau)^{p-1}}{\tau^{2}}
    \\
    &\quad\times
      \big[
      p\tau
      \der{}{\tau}\big(\excf^{(-1)}(\tau)\big)
      -
      \excf^{(-1)}(\tau)
      \big]
      \,.
  \end{split}
\end{equation}
Denote next $y=\excf^{(-1)}(\tau)$; then from the elementary formula for the derivative of the inverse function we get
\begin{equation}
  \label{eq:sss_ups_ttar2}
  \frac{\excf^{(-1)}(\tau)^{p-1}}{\tau^{2}}
      \big[
      p\tau
      \der{}{\tau}\big(\excf^{(-1)}(\tau)\big)
      -
      \excf^{(-1)}(\tau)
      \big]
      =
      \frac{y^{p-1}}{\tau^{2}}
      \Big[
      \frac{p\excf(y)}{\excf'(y)}
      -
      y
      \Big]
      \,,
\end{equation}
and, from \eqref{eq:sss_excf_nn},
\begin{equation}
  \label{eq:sss_ups_ttar3}
  y^{p}
  \Big(
  \frac{p}{\alpha_{1}}
  -
  1
  \Big)
  \ge
  y^{p-1}
  \Big[
  \frac{p\excf(y)}{\excf'(y)}
  -
  y
  \Big]
  \ge
  y^{p}
  \Big(
  \frac{p}{\alpha_{2}}
  -
  1
  \Big)
  \,.
\end{equation}
On combining \eqref{eq:sss_ups_ttar}--\eqref{eq:sss_ups_ttar3} we obtain, after using again the definition of $\tar(\tau)$,
\begin{align}
  \label{eq:sss_ups_ttarb}
  \der{}{t}
  \big(\tar(\tau)^{\frac{1}{p-1}}\big)
  &\ge
    \frac{p-\alpha_{2}}{\alpha_{2}(p-1)}
    \frac{\varGamma}{t+t_{0}}
    \frac{1}{\tau}
    \tar(\tau)^{\frac{1}{p-1}}
    \,,
  \\
  \label{eq:sss_ups_ttara}
  \der{}{t}
  \big(\tar(\tau)^{\frac{1}{p-1}}\big)
  &\le
    \frac{p-\alpha_{1}}{\alpha_{1}(p-1)}
    \frac{\varGamma}{t+t_{0}}
    \frac{1}{\tau}
    \tar(\tau)^{\frac{1}{p-1}}
    \,.
\end{align}

In order for $\ups$ to be a supersolution, we need to show
\begin{equation}
  \label{eq:sss_ups_ineq}
  \pder{\ups}{t}
  \ge
  \pder{I_{1}}{\rpol}
  +
  I_{1}
  \Big(
  \frac{N-1}{\rpol}
  +
  \exw'
  \Big)
  \,,
  \qquad
  \rpol>\rpol_{0}
  \,.
\end{equation}
We read the three terms appearing in \eqref{eq:sss_ups_ineq}, in order, in \eqref{eq:sss_ups_tder}, in \eqref{eq:sss_ups_diff2}--\eqref{eq:sss_ups_diff3}, in \eqref{eq:sss_ups_diff}. We note that all these quantities share the common factor
\begin{equation*}
  \frac{
    C_{*}
    \brt(\rpol,\tau)^{\frac{p-1}{p+m-3}-1}
  }{
    (p+m-3)(t+t_{0})^{\frac{p+m-2}{p+m-3}}
  }
  \,.
\end{equation*}
Upon dividing \eqref{eq:sss_ups_ineq} by this factor, we obtain its
equivalent form
\begin{equation}
  \label{eq:sss_ups_ineq2}
  \begin{split}
    K_{0}&:=
      (p-1)
      (t+t_{0})
      \der{}{t}
      \big(
      \tar(\tau)^{\frac{1}{p-1}}
      \big)
      \ge
      \brt(\rpol,\tau)
      \\
    &\quad
      -
      \frac{
      C_{*}^{p+m-3}\brt(\rpol,\tau)
      }{
      (p+m-3)^{p-2}
      }
      \rat(\rpol)^{2-p}
      \rat'(\rpol)^{p-1}
      \Big(
      \frac{N-1}{\rpol}
      +
      \exw'(\rpol)
      \Big)
      \\
    &\quad
      -
      \frac{
      C_{*}^{p+m-3} I_{3}
      }{
      (p+m-3)^{p-2}
      }
      =:
      \brt(\rpol,\tau)
      +
      K_{1}
      +
      K_{2}
      \,,
  \end{split}
\end{equation}
which we may rewrite as
\begin{equation*}
  K_{0}
  -
  K_{1}
  \ge
  A
  +
  K_{2}
  \,.
\end{equation*}
Let us start to estimate the quantities of interest.
\\
We have from \eqref{eq:sss_ups_ttarb}
\begin{equation}
  \label{eq:sss_ups_k0}
  K_{0}
  \ge
  \frac{p-\alpha_{2}}{\alpha_{2}}
  \frac{1}{\log(t+t_{0})}
  \tar(\tau)^{\frac{1}{p-1}}
  \,.
\end{equation}
Then, from \eqref{eq:sss_rat_n} we infer, on using also the definition of $\rat(\rpol)$,
\begin{equation}
  \label{eq:sss_ups_k1}
  \begin{split}
    -K_{1}
    &\ge
      \frac{
      C_{*}^{p+m-3}\brt(\rpol,\tau)
      }{
      (p+m-3)^{p-2}
      }
      (p-\alpha_{2})^{p-1}
      \frac{\rpol}{\excf(\rpol)}
      \Big(
      \frac{N-1}{\rpol}
      +
      \exw'(\rpol)
      \Big)
      \\
    &\ge
      \frac{
      C_{*}^{p+m-3}\brt(\rpol,\tau)
      }{
      (p+m-3)^{p-2}
      }
      (p-\alpha_{2})^{p-1}
      \Big(
      \frac{N-1}{\excf(\rpol)}
      +
      \alpha_{1}^{2}
      \Big)
      \\
    &=
      \mu_{1}
      C_{*}^{p+m-3}
      \brt(\rpol,\tau)
      \Big(
      \frac{N-1}{\excf(\rpol)}
      +
      \alpha_{1}^{2}
      \Big)
      \,,
  \end{split}
\end{equation}
where we used also \eqref{eq:sss_excf_nn}, and set
\begin{equation}
  \label{eq:sss_ups_mu1}
  \mu_{1}
  =
  \frac{(p-\alpha_{2})^{p-1}}{(p+m-3)^{p-2}}
  \,.
\end{equation}
Next we estimate from above $K_{2}=:h_{1}+h_{2}$, where (cf \eqref{eq:sss_ups_diff3})
\begin{equation}
  \label{eq:sss_ups_h1}
  \begin{split}
    h_{1}
    &:=
      \frac{
      C_{*}^{p+m-3}
      }{
      (p+m-3)^{p-2}
      }
      \frac{1}{p+m-3}
      \rat(\rpol)^{-\frac{p(p-2)}{p-1}}
      \rat'(\rpol)^{p}
    \\
    &\le
      \frac{
      C_{*}^{p+m-3}
      (p-\alpha_{1})^{p}
      }{
      (p+m-3)^{p-1}
      }
      \rat(\rpol)^{\frac{1}{p-1}}
      \frac{1}{\excf(\rpol)}
    \\
    &=
      \mu_{2}
      C_{*}^{p+m-3}
      \rat(\rpol)^{\frac{1}{p-1}}
      \frac{1}{\excf(\rpol)}
      \,;
  \end{split}
\end{equation}
we used here \eqref{eq:sss_rat_n} and the definition of $\rat(\rpol)$, setting also
\begin{equation}
  \label{eq:sss_ups_mu2}
  \mu_{2}
  =
  \frac{(p-\alpha_{1})^{p}}{(p+m-3)^{p-1}}
  \,.
\end{equation}
Then we calculate (see again \eqref{eq:sss_ups_diff3})
\begin{equation}
  \label{eq:sss_ups_h2}
  \begin{split}
    h_{2}
    &:=
      \frac{
      C_{*}^{p+m-3}
      }{
      (p+m-3)^{p-2}
      }
      \brt(\rpol,\tau)
      \big[
      (p-2)
      \rat(\rpol)^{1-p}
      \rat'(\rpol)^{p}
      \\
    &\quad
      -
      (p-1)
      \rat(\rpol)^{2-p}
      \rat'(\rpol)^{p-2}
      \rat''(\rpol)
      \big]
      \\
    &\le
      \frac{
      C_{*}^{p+m-3}
      \brt(\rpol,\tau)
      }{
      (p+m-3)^{p-2}
      }
      \big[
      (p-2)_{+}
      (p-\alpha_{1})^{p}
      +
      d
      \big]
      \frac{1}{\excf(\rpol)}
      \\
    &=
      \mu_{3}
      C_{*}^{p+m-3}
      \frac{\brt(\rpol,\tau)}{\excf(\rpol)}
      \,,
  \end{split}
\end{equation}
where we used \eqref{eq:sss_rat_n}, \eqref{eq:sss_rat_nn}, and set
\begin{gather}
  \label{eq:sss_ups_d}
  d
  =
  (c_{1})_{-}
  (p-1)
  (p-\alpha_{i})^{p-2}
  \,,
  \quad
  \text{$i=1$ if $p\ge 2$, $i=2$ if $p<2$;}
  \\
  \label{eq:sss_ups_mu3}
  \mu_{3}
  =
  \frac{
    (p-2)_{+}
    (p-\alpha_{1})^{p}
    +
    d
  }{
    (p+m-3)^{p-2}
  }
  \,.
\end{gather}
Namely, $h_{2}$ does not give any contribution if $p\le2$ and $c_{1}\ge 0$. 

We collect \eqref{eq:sss_ups_k0}--\eqref{eq:sss_ups_h2} and see that \eqref{eq:sss_ups_ineq2} is implied by (leaving on the right hand side of \eqref{eq:sss_ups_ineq3} only the contribution of $h_{1}$)
\begin{equation}
  \label{eq:sss_ups_ineq3}
  \begin{split}
    &C_{*}^{p+m-3}
      \brt(\rpol,\tau)
      \Big(
      \frac{\mu_{1}(N-1)-\mu_{3}}{\excf(\rpol)}
      +
      \mu_{1}
      \alpha_{1}^{2}
      -
      C_{*}^{-(p+m-3)}
      \Big)
    \\
    &\quad+
      \frac{p-\alpha_{2}}{\alpha_{2}}
      \frac{1}{\log(t+t_{0})}
      \tar(\tau)^{\frac{1}{p-1}}
     \\
    &\ge
      \mu_{2}
      C_{*}^{p+m-3}
      \rat(\rpol)^{\frac{1}{p-1}}
      \frac{1}{\excf(\rpol)}
      \,.    
  \end{split}
\end{equation}
The constants $\mu_{1}>0$, $\mu_{2}>0$ and $\mu_{3}\ge 0$ are defined in \eqref{eq:sss_ups_mu1}, \eqref{eq:sss_ups_mu2} and \eqref{eq:sss_ups_d}--\eqref{eq:sss_ups_mu3}, respectively, and depend only on $p$, $m$, $\alpha_{1}$, $\alpha_{2}$.
We select $C_{*}$ and $\rpol_{0}>0$ large enough to have
\begin{gather}
  \label{eq:sss_ups_r0}
  \excf(\rpol_{0})
  \ge
  \max\Big\{
  \frac{4\mu_{3}}{\mu_{1}\alpha_{1}^{2}}
  \,,
  \frac{2\mu_{2}p(\alpha_{2}-1)_{+}}{\mu_{1}\alpha_{1}^{2}(p-\alpha_{2})}
  \Big\}
  \,;
  \\
  \label{eq:sss_ups_C}
  C_{*}^{-(p+m-3)}
  \le
  \min\Big\{
  \frac{\mu_{1}\alpha_{1}^{2}}{4}
  \,,
  \frac{(N-1)(p-\alpha_{2})^{p-1}}{(p+m-3)^{p-2}\excf(\rpol_{0})}
  \Big\}
  \,.
\end{gather}
Under the assumptions \eqref{eq:sss_ups_r0}--\eqref{eq:sss_ups_C} (actually we exploit here only the first terms in the $\max$ functions), \eqref{eq:sss_ups_ineq3} is implied by
\begin{equation}
  \label{eq:sss_ups_ineq4}
  \begin{split}
    & \mu_{1}\alpha_{1}^{2}
      \brt(\rpol,\tau)
      +
      \frac{2(p-\alpha_{2})}{C_{*}^{p+m-3}\alpha_{2}}
      \frac{1}{\log(t+t_{0})}
      \tar(\tau)^{\frac{1}{p-1}}
     \\
    &\ge
      2\mu_{2}
      \rat(\rpol)^{\frac{1}{p-1}}
      \frac{1}{\excf(\rpol)}
      \,.    
  \end{split}
\end{equation}
Recalling the definitions of $\brt(\rpol,\tau)$ and of $\tau$, we see that \eqref{eq:sss_ups_ineq4} is equivalent to
\begin{equation}
  \label{eq:sss_ups_ineq5}
  \begin{split}
    & \tar(\tau)^{\frac{1}{p-1}}
      \Big[
      \mu_{1}\alpha_{1}^{2}
      +
      \varGamma
      \frac{2(p-\alpha_{2})}{C_{*}^{p+m-3}\alpha_{2}}
      \frac{1}{\tau}
      \Big]
     \\
    &\ge
      \rat(\rpol)^{\frac{1}{p-1}}
      \Big[
      \mu_{1}\alpha_{1}^{2}
      +
      2\mu_{2}
      \frac{1}{\excf(\rpol)}
      \Big]
      =:
      \auxf(\rpol)
      \,.    
  \end{split}
\end{equation}
We select here
\begin{equation}
  \label{eq:sss_ups_Gamma0}
  \varGamma
  \ge
  \max\Big(
  \frac{\mu_{2}\alpha_{2}C_{*}^{p+m-3}}{p-\alpha_{2}}
  ,
  1
  \Big)
  \,,
\end{equation}
so that (actually by exploiting only the first term in the $\max$ function in \eqref{eq:sss_ups_Gamma0}) clearly the left hand side of \eqref{eq:sss_ups_ineq5} is greater than or equal to $\auxf(\excf^{(-1)}(\tau))$, and the inequality \eqref{eq:sss_ups_ineq5} follows from
\begin{equation}
  \label{eq:sss_ups_ineq5a}
  \auxf(\excf^{(-1)}(\tau))
  \ge
  \auxf(\rpol)
  \,.
\end{equation}
However, if $\brt(\rpol,\tau)>0$, we are in the region
$\tar(\tau)>\rat(\rpol)$, implying $\rpol<\excf^{(-1)}(\tau)$. Then, on invoking
Lemma~\ref{l:sss_comb} with
\begin{equation}
  \label{eq:sss_ups_deltas}
  \delta_{1}
  =
  \mu_{1}\alpha_{1}^{2}
  \,,
  \qquad
  \delta_{2}
  =
  2\mu_{2}
  \,,
\end{equation}
\eqref{eq:sss_ups_ineq5a} follows from $\rpol_{0}\le\rpol\le \excf^{(-1)}(\tau)$ and from our choice of $\rpol_{0}$ (such that $\rpol_{0}\ge \rpol_{1}$ as in \eqref{eq:sss_comb_n}; we exploit here only the second term in the $\max$ function in \eqref{eq:sss_ups_r0}).

2) Case $\rpol<\rpol_{0}$.
With the notation already introduced, we write for $\rpol<\rpol_{0}$
\begin{equation}
  \label{eq:sss_sup_3}
  \ups(x,t)
  =
  \frac{C_{*}\trb(\rpol,\tau)^{\frac{p-1}{p+m-3}}}{(t+t_{0})^{\frac{1}{p+m-3}}}
  \,,
  \quad
  \trb(\rpol,\tau)
  :=
  \tar(\tau)^{\frac{1}{p-1}}
  -
  \ram(\rpol)
  \,,
\end{equation}
and recall that $\trb(\rpol_{0},\tau)=\brt(\rpol_{0},\tau)$. Then the continuity of
$\ups_{\rpol}$ at $\rpol=\rpol_{0}$ amounts to selecting $\nu_{0}$ so that
\begin{equation*}
  \begin{split}
    \nu_{0}
    \frac{p}{p-1}
    \frac{\rpol_{0}^{\frac{1}{p-1}}}{\excf(\rpol_{0})^{\frac{1}{p-1}}}
    &=
      \ram'(\rpol_{0})
      =
      \der{}{\rpol}
      (\rat^{\frac{1}{p-1}})
      (\rpol_{0})
      =
      \frac{1}{p-1}
      \rat(\rpol_{0})^{\frac{1}{p-1}-1}
      \rat'(\rpol_{0})
    \\
    &=
      \frac{\rat(\rpol_{0})^{\frac{1}{p-1}}}{(p-1)\rpol_{0}}
      \Big[
      p
      -
      \frac{\rpol_{0}\excf'(\rpol_{0})}{\excf(\rpol_{0})}
      \Big]
      \,,
  \end{split}
\end{equation*}
that is
\begin{equation}
  \label{eq:sss_ups_c1}
  0<
  1
  -
  \frac{\alpha_{2}}{p}
  \le
  \nu_{0}
  =
  1
  -
  \frac{\rpol_{0}\excf'(\rpol_{0})}{p\excf(\rpol_{0})}
  \le
  1
  -
  \frac{\alpha_{1}}{p}
  <
  1
  \,,
\end{equation}
according to \eqref{eq:sss_excf_nn} and our assumption $\alpha_{2}<p$.

We work in the set $\trb>0$. Then we compute
\begin{equation}
  \label{eq:sss_ups_diffs}
  \begin{split}
    L_{1}
    &:=
      \ups^{m-1}
      \abs{\ups_{\rpol}}^{p-2}
      \ups_{\rpol}
    \\
    &=
      -
      \frac{
      C_{*}^{p+m-2}
      (\nu_{0}p)^{p-1}
      }{
      (p+m-3)^{p-1}
      }
      \,
      \frac{
      \trb(\rpol,\tau)^{\frac{p-1}{p+m-3}}
      }{
      (t+t_{0})^{\frac{p+m-2}{p+m-3}}
      }
      \frac{\rpol}{\excf(\rpol_{0})}
      \,.
  \end{split}
\end{equation}
We recall that, in order for $\ups$ to be a supersolution, we need to check (see also \eqref{eq:sss_ups_ineq})
\begin{equation}
  \label{eq:sss_ups_ineqs}
  \pder{\ups}{t}
  \ge
  \pder{L_{1}}{\rpol}
  +
  L_{1}
  \Big(
  \frac{N-1}{\rpol}
  +
  \exw'(\rpol)
  \Big)
  \,,
  \qquad
  \rpol<\rpol_{0}
  \,.
\end{equation}
The time derivative $\ups_{t}$ is given by \eqref{eq:sss_ups_tder},
when we formally substitute $\brt(\rpol,\tau)$ with $\trb(\rpol,\tau)$.
Next we calculate
\begin{equation}
  \label{eq:sss_ups_diffs2}
  \begin{split}
    \pder{L_{1}}{\rpol}
    &=
      \frac{
      C_{*}^{p+m-2}
      (\nu_{0}p)^{p-1}
      }{
      (p+m-3)^{p-1}
      }
      \,
      \frac{
      \trb(\rpol,\tau)^{\frac{p-1}{p+m-3}-1}
      }{
      (t+t_{0})^{\frac{p+m-2}{p+m-3}}
      \excf(\rpol_{0})
      }
    \\
    &\quad\times
      \Big[
      \frac{p\nu_{0}}{p+m-3}
      \frac{\rpol^{\frac{p}{p-1}}}{\excf(\rpol_{0})^{\frac{1}{p-1}}}
      -
      \trb(\rpol,\tau)
      \Big]
      \,.
  \end{split}
\end{equation}
Upon dividing both sides of \eqref{eq:sss_ups_ineqs} by
\begin{equation*}
  \frac{C_{*}}{p+m-3}
  \,
  \frac{\trb(\rpol,\tau)^{\frac{p-1}{p+m-3}-1}}{(t+t_{0})^{\frac{p+m-2}{p+m-3}}}
  \,,
\end{equation*}
we obtain the equivalent version of \eqref{eq:sss_ups_ineqs} (here $K_{0}$ is as in \eqref{eq:sss_ups_ineq2})
\begin{equation}
  \label{eq:sss_ups_ineqs2}
  \begin{split}
    K_{0}&=
           (p-1)
           (t+t_{0})
           \der{}{t}
           \big(
           \tar(\tau)^{\frac{1}{p-1}}
           \big)
           \ge
           \trb(\rpol,\tau)
    \\
         &\quad+
           \frac{C_{*}^{p+m-3}(\nu_{0}p)^{p-1}}{(p+m-3)^{p-2}}
           \frac{1}{\excf(\rpol_{0})}
           \Big[
           \frac{p\nu_{0}}{p+m-3}
           \frac{\rpol^{\frac{p}{p-1}}}{\excf(\rpol_{0})^{\frac{1}{p-1}}}
           -
           \trb(\rpol,\tau)
           \Big]
    \\
         &\quad
           -
           \Big(
           \frac{N-1}{\rpol}
           +
           \exw'(\rpol)
           \Big)
           \frac{C_{*}^{p+m-3}(\nu_{0}p)^{p-1}}{(p+m-3)^{p-2}}
           \frac{\trb(\rpol,\tau)\rpol}{\excf(\rpol_{0})}
           =:
           K_{3}
           \,.
  \end{split}
\end{equation}
On recalling \eqref{eq:sss_ups_k0}, we see that $K_{0}>0$, so that we need only show
\begin{equation}
  \label{eq:sss_ups_ineqs3}
      0
      \ge
      K_{3}
      \,.
\end{equation}
Clearly, this is implied by
\begin{equation}
  \label{eq:sss_ups_ineqs4}
  \begin{split}
    & \Big(
      \frac{(N-1)C_{*}^{p+m-3}(\nu_{0}p)^{p-1}}{(p+m-3)^{p-2}\excf(\rpol_{0})}
      -
      1
      \Big)
      \trb(\rpol,\tau)
      \ge
    \\
    &\quad
      \frac{C_{*}^{p+m-3}(\nu_{0}p)^{p-1}}{(p+m-3)^{p-2}}
      \frac{1}{\excf(\rpol_{0})}
      \Big[
      \frac{p\nu_{0}}{p+m-3}
      \frac{\rpol^{\frac{p}{p-1}}}{\excf(\rpol_{0})^{\frac{1}{p-1}}}
      -
      \trb(\rpol,\tau)
      \Big]
      \,.
  \end{split}
\end{equation}
The left hand side of \eqref{eq:sss_ups_ineqs4} is nonnegative, owing to our choice of $C_{*}$ and to \eqref{eq:sss_ups_c1} (we use the second term in the $\max$ function in \eqref{eq:sss_ups_C} here). Thus we only need to show that the right hand side there is nonpositive for $\rpol<\rpol_{0}$. According to the definitions of $\tau$ and of $\trb(\rpol,\tau)$ this in turn follows from
\begin{equation}
  \label{eq:sss_ups_ineqs5}
  \begin{split}
    &\frac{p\nu_{0}}{p+m-3}
      \frac{\rpol_{0}^{\frac{p}{p-1}}}{\excf(\rpol_{0})^{\frac{1}{p-1}}}
      +
      \ram(\rpol_{0})
      \le
      \tar(\log t_{0})^{\frac{1}{p-1}}
    \\
    &\quad\le
      \tar(\varGamma \log t_{0})^{\frac{1}{p-1}}
      \le
      \tar(\tau)^{\frac{1}{p-1}}
      \,,
  \end{split}
\end{equation}
where the last two inequalities follow trivially from the increasing character of $\tar(\tau)$ and $\varGamma\ge 1$. But since $\rpol_{0}$ has been fixed, this amounts simply to choosing $t_{0}=t_{0}(\rpol_{0})$ so that both \eqref{eq:sss_ups_t0r0} and \eqref{eq:sss_ups_ineqs5} are satisfied; actually it is easily seen that \eqref{eq:sss_ups_t0r0} follows from \eqref{eq:sss_ups_ineqs5}. Note that $r_{0}$, $C_{*}$ and $t_{0}$ are independent of $\varGamma$.

3) Finally, the regularity required in Definition~\ref{d:weaksol} can be easily shown by direct inspection and from the explicit form of $\ups_{t}$ found above.

We conclude this Section with the following result, which is instrumental in the proof of Theorem~\ref{t:sup_sup}.

\begin{lemma}
  \label{l:sss_comp}
  Let $L>0$, $M>0$ be given positive numbers. Then it is possible to select $\rpol_{0}$, $C_{*}$, $t_{0}$, $\varGamma$ so that $\ups$ is a supersolution and 
  \begin{equation}
    \label{eq:sss_comp_n}
    \ups(x,0)
    \ge
    M
    \,,
    \qquad
    \abs{x}\le L
    \,.
  \end{equation}
\end{lemma}

\begin{proof}
  First we choose above $\rpol_{0}$ satisfying both
  \eqref{eq:sss_ups_r0} and $\rpol_{0}\ge L$. Then we choose $C_{*}$
  as in \eqref{eq:sss_ups_C}, and $t_{0}$ according to
  \eqref{eq:sss_ups_ineqs5}. Next we remark that, owing to
  \eqref{eq:sss_powerlike_gt_excf}, we have for $\rpol\le \rpol_{0}$
  \begin{equation}
    \label{eq:sss_comp_i}
    \begin{split}
      \ups(x,0)
      &\ge
        C_{*}
        t_{0}^{-\frac{1}{p+m-3}}
        \Big[
        \tar(\varGamma \log t_{0})^{\frac{1}{p-1}}
        -
        \rat(\rpol_{0})^{\frac{1}{p-1}}
        \Big]^{\frac{p-1}{p+m-3}}
        \\
      &\ge
        C_{*}
        t_{0}^{-\frac{1}{p+m-3}}
        \Big[
        (\varGamma \log t_{0})^{\frac{p-\alpha_{2}}{\alpha_{2}(p-1)}}
        -
        \rat(\rpol_{0})^{\frac{1}{p-1}}
        \Big]^{\frac{p-1}{p+m-3}}
        \ge
        M
        \,,
    \end{split}
  \end{equation}
  where last inequality is guaranteed by a suitable choice of $\varGamma$, of course preserving \eqref{eq:sss_ups_Gamma0}.
\end{proof}

\section{Self similar subsolutions}
\label{s:ssb}

We look for subsolutions in the radial form
\eqref{eq:sss_sup}--\eqref{eq:sss_ups_2}. We use here the same
notation introduced in Section~\ref{s:sss}, and we look at the set
where $\ups>0$. In this Section, the choice of the constants is more
delicate; we find it convenient to establish it from the beginning. In
this connection, let us note that the constant $\tilde\mu_{1}>0$ is
defined in \eqref{eq:sss_ubs_mu1}, $\tilde\mu_{2}>0$ in
\eqref{eq:sss_ubs_mu2}, and $\tilde\mu_{3}\ge 0$ in
\eqref{eq:sss_ubs_mu3}; the $\tilde\mu_{i}$ depend on $p$, $m$,
$\alpha_{1}$, $\alpha_{2}$ only. The constant $\nu_{0}$ is chosen as
in \eqref{eq:sss_ups_c1}. We also introduce for future reference a
given constant $\lambda>0$.
\\
We fix $t_{0}>1$ such that
\begin{equation}
  \label{eq:sss_ubs_t0}
  \begin{split}
    \log t_{0}
    \ge
      \max&\Big\{
      4
      \frac{p-\alpha_{1}}{\alpha_{1}}
      \,,
      \tilde\mu_{4}^{-1}
      \,,
      \frac{2\tilde\mu_{1}(N-1)+\tilde\mu_{3}
      +
      2\tilde\mu_{1}
      \alpha_{2}^{2}}
      {\tilde\mu_{4}}
      \,,
    \\
    &\quad
      \frac{2(\nu_{0}p)^{p-1}}{\tilde\mu_{4}(p+m-3)^{p-2}}
      (
      N
      +
      \alpha_{2}^{2}
      )
      \Big\}
      \,,
  \end{split}
\end{equation}
where
\begin{equation}
  \label{eq:ssb_ups_mu4}
  \tilde\mu_{4}
  =
  2^{-\frac{\alpha_{2}(p-1)}{p-\alpha_{2}}}
  \frac{\tilde\mu_{2}\alpha_{1}}{p-\alpha_{1}}
  >0
  \,.  
\end{equation}
Next we select $\rpol_{0}>0$ such that
\begin{equation}
  \label{eq:ssb_ups_r0}
  \excf(r_{0})
  <
  \min
  \Big\{
  1
  \,,
  \lambda^{p+m-3}
  \,,
  \tilde\mu_{4}
  C_{*}^{p+m-3}
  \log t_{0}
  \Big\}
  \,,
\end{equation}
where we set
\begin{equation}
  \label{eq:ssb_ups_C}
  \begin{split}
  C_{*}^{-(p+m-3)}
  &=
    \max\Big\{
    \lambda^{-(p+m-3)}
    \,,
    2\frac{\tilde\mu_{1}(N-1)+\tilde\mu_{3}}{\excf(\rpol_{0})}
    +
    2\tilde\mu_{1}
    \alpha_{2}^{2}
    \,,
  \\
  &\quad
    \frac{2(\nu_{0}p)^{p-1}}{(p+m-3)^{p-2}}
    \frac{1}{\excf(\rpol_{0})}
    (
    N
    +
    \alpha_{2}^{2}
    \excf(\rpol_{0})
    )
    \Big\}
    \,.
  \end{split}
\end{equation}
Clearly we have to show that \eqref{eq:ssb_ups_r0}, \eqref{eq:ssb_ups_C} are compatible, that is that 
\begin{equation}
  \label{eq:ssb_ups_r0_i}
  \excf(\rpol_{0})C_{*}^{-(p+m-3)}
  <
  \tilde\mu_{4}
  \log t_{0}
\end{equation}
can be fulfilled by choosing a small enough $\rpol_{0}$. Indeed, also invoking $\excf(\rpol_{0})<\min\{1,\lambda^{p+m-3}\}$ which is certainly meaningful, we compute from \eqref{eq:ssb_ups_C}
\begin{equation}
  \label{eq:ssb_ups_r0_ii}
  \begin{split}
  \excf(\rpol_{0})C_{*}^{-(p+m-3)}
  &<
    \max\Big\{
    1
    \,,
    2\tilde\mu_{1}(N-1)+\tilde\mu_{3}
    +
    2\tilde\mu_{1}
    \alpha_{2}^{2}
    \,,
  \\
  &\quad
    \frac{2(\nu_{0}p)^{p-1}}{(p+m-3)^{p-2}}
    (
    N
    +
    \alpha_{2}^{2}
    )
    \Big\}
    \le
    \tilde\mu_{4}
    \log t_{0}
    \,,
  \end{split}
\end{equation}
according to our choice of $t_{0}$.
Finally we select
\begin{equation}
  \label{eq:ssb_ups_Gamma0}
  \varGamma
  =
  2^{\frac{\alpha_{2}(p-1)}{p-\alpha_{2}}}
  \frac{\excf(\rpol_{0})}{\log t_{0}}
  \,.
\end{equation}

1) Case $\rpol>\rpol_{0}$.
In order for $\ups$ to be a subsolution, we need to show
\begin{equation}
  \label{eq:ssb_ups_ineq}
  \pder{\ups}{t}
  \le
  \pder{I_{1}}{\rpol}
  +
  I_{1}
  \Big(
  \frac{N-1}{\rpol}
  +
  \exw'
  \Big)
  \,,
\end{equation}
which, exactly as in \eqref{eq:sss_ups_ineq2}, we may reduce to the equivalent form
\begin{equation}
  \label{eq:ssb_ups_ineq2a}
  K_{0}
  -
  K_{1}
  \le
  A
  +
  K_{2}
  \,.
\end{equation}
We have from \eqref{eq:sss_ups_ttara}
\begin{equation}
  \label{eq:ssb_ups_k0}
  K_{0}
  \le
  \frac{p-\alpha_{1}}{\alpha_{1}}
  \frac{1}{\log(t+t_{0})}
  \tar(\tau)^{\frac{1}{p-1}}
  \,.
\end{equation}
Next, from \eqref{eq:sss_rat_n} we infer, on using also the definition of $\rat(\rpol)$,
\begin{equation}
  \label{eq:ssb_ups_k1}
  \begin{split}
    -K_{1}
    &\le
      \frac{
      C_{*}^{p+m-3}\brt(\rpol,\tau)
      }{
      (p+m-3)^{p-2}
      }
      (p-\alpha_{1})^{p-1}
      \frac{\rpol}{\excf(\rpol)}
      \Big(
      \frac{N-1}{\rpol}
      +
      \exw'(\rpol)
      \Big)
      \\
    &\le
      \frac{
      C_{*}^{p+m-3}\brt(\rpol,\tau)
      }{
      (p+m-3)^{p-2}
      }
      (p-\alpha_{1})^{p-1}
      \Big(
      \frac{N-1}{\excf(\rpol)}
      +
      \alpha_{2}^{2}
      \Big)
      \\
    &=
      \tilde\mu_{1}
      C_{*}^{p+m-3}
      \brt(\rpol,\tau)
      \Big(
      \frac{N-1}{\excf(\rpol)}
      +
      \alpha_{2}^{2}
      \Big)
      \,,
  \end{split}
\end{equation}
on invoking also \eqref{eq:sss_excf_nn}, and setting
\begin{equation}
  \label{eq:sss_ubs_mu1}
  \tilde\mu_{1}
  =
  \frac{
    (p-\alpha_{1})^{p-1}
  }{
    (p+m-3)^{p-2}
  }
  \,.
\end{equation}
Next we estimate from below $K_{2}=h_{1}+h_{2}$ (see \eqref{eq:sss_ups_h1} and \eqref{eq:sss_ups_h2}). Namely
\begin{equation}
  \label{eq:ssb_ups_h1}
  \begin{split}
    h_{1}
    &\ge
      \frac{
      C_{*}^{p+m-3}
      (p-\alpha_{2})^{p}
      }{
      (p+m-3)^{p-1}
      }
      \rat(\rpol)^{\frac{1}{p-1}}
      \frac{1}{\excf(\rpol)}
    \\
    &=
      \tilde\mu_{2}
      C_{*}^{p+m-3}
      \rat(\rpol)^{\frac{1}{p-1}}
      \frac{1}{\excf(\rpol)}
      \,,
  \end{split}
\end{equation}
where we appealed to \eqref{eq:sss_rat_n} and to the definition of $\rat(\rpol)$, and we also set
\begin{equation}
  \label{eq:sss_ubs_mu2}
  \tilde\mu_{2}
  =
  \frac{
    (p-\alpha_{2})^{p}
  }{
    (p+m-3)^{p-1}
  }
  \,.
\end{equation}
Then we calculate, reasoning as in \eqref{eq:sss_ups_h2},
\begin{equation}
  \label{eq:ssb_ups_h2}
  \begin{split}
    h_{2}
    &\ge
      \frac{
      C_{*}^{p+m-3}
      \brt(\rpol,\tau)
      }{
      (p+m-3)^{p-2}
      }
      \big[
      -
      (p-2)_{-}
      (p-\alpha_{1})^{p}
      -
      \tilde d
      \big]
      \frac{1}{\excf(\rpol)}
      \\
    &=
      -
      \tilde\mu_{3}
      C_{*}^{p+m-3}
      \frac{\brt(\rpol,\tau)}{\excf(\rpol)}
      \,,
  \end{split}
\end{equation}
where we set
\begin{gather}
  \label{eq:sss_ubs_d}
  \tilde d
  =
  (c_{2})_{+}
  (p-1)
  (p-\alpha_{i})^{p-2}
  \,,
  \quad
  \text{$i=1$ if $p\ge 2$, $i=2$ if $p<2$,}
  \\
  \label{eq:sss_ubs_mu3}
  \tilde\mu_{3}
  =
  \frac{
    (p-2)_{-}
    (p-\alpha_{1})^{p}
    +
    \tilde d}{
    (p+m-3)^{p-2}
  }
  \,.
\end{gather}
We remark that $h_{2}$ does not give any contribution if $p\ge2$ and $c_{2}\le 0$. 

We collect \eqref{eq:ssb_ups_k0}--\eqref{eq:ssb_ups_h2} and see that \eqref{eq:ssb_ups_ineq2a} is implied by
\begin{equation}
  \label{eq:ssb_ups_ineq3}
  \begin{split}
    &C_{*}^{p+m-3}
      \brt(\rpol,\tau)
      \Big(
      \frac{\tilde\mu_{1}(N-1)+\tilde\mu_{3}}{\excf(\rpol)}
      +
      \tilde\mu_{1}
      \alpha_{2}^{2}
      -
      C_{*}^{-(p+m-3)}
      \Big)
    \\
    &\quad+
      \frac{p-\alpha_{1}}{\alpha_{1}}
      \frac{1}{\log(t+t_{0})}
      \tar(\tau)^{\frac{1}{p-1}}
     \\
    &\le
      \tilde\mu_{2}
      C_{*}^{p+m-3}
      \rat(\rpol)^{\frac{1}{p-1}}
      \frac{1}{\excf(\rpol)}
      \,.    
  \end{split}
\end{equation}
As in \eqref{eq:sss_ups_ineq3}, we leave on the right hand side of
\eqref{eq:ssb_ups_ineq3} only the contribution of $h_{1}$. The
constants $\tilde\mu_{1}>0$, $\tilde\mu_{2}>0$ and $\tilde\mu_{3}\ge 0$ are defined in
\eqref{eq:sss_ubs_mu1}, \eqref{eq:sss_ubs_mu2} and
\eqref{eq:sss_ubs_mu3}, respectively, and depend only on $p$, $m$,
$\alpha_{1}$, $\alpha_{2}$.

Under assumption \eqref{eq:ssb_ups_C} (where we exploit only the second term in the $\max$ function), \eqref{eq:ssb_ups_ineq3} is implied by
\begin{equation}
  \label{eq:ssb_ups_ineq4}
  \begin{split}
    & -
      \frac{1}{2}
      \brt(\rpol,\tau)
      +
      \frac{p-\alpha_{1}}{\alpha_{1}}
      \frac{1}{\log(t+t_{0})}
      \tar(\tau)^{\frac{1}{p-1}}
     \\
    &\le
      \tilde\mu_{2}
      C_{*}^{p+m-3}
      \rat(\rpol)^{\frac{1}{p-1}}
      \frac{1}{\excf(\rpol)}
      \,,
  \end{split}
\end{equation}
that is, owing to the definitions of $\brt(\rpol,\tau)$ and of $\tau$,
\begin{equation}
  \label{eq:ssb_ups_ineq5}
  \begin{split}
    &\tar(\tau)^{\frac{1}{p-1}}
    \Big[
    1
    -
    \varGamma
    \frac{2(p-\alpha_{1})}{\alpha_{1}}
    \frac{1}{\tau}
    \Big]
      \ge
    \\
    &\quad
    \rat(\rpol)^{\frac{1}{p-1}}
    \Big[
    1
    -
    2\tilde\mu_{2}
    C_{*}^{p+m-3}
    \frac{1}{\excf(\rpol)}
    \Big]
    \,.
  \end{split}
\end{equation}
However, as in Section~\ref{s:sss} we note that, if $\brt(\rpol,\tau)>0$, then
$\tar(\tau)>\rat(\rpol)$, so that $\rpol<\excf^{(-1)}(\tau)$. Thus, elementarily, \eqref{eq:ssb_ups_ineq5} follows if
\begin{equation*}
  \varGamma
  \frac{p-\alpha_{1}}{\alpha_{1}}
  \le
  \tilde\mu_{2}
  C_{*}^{p+m-3}
  \,,
\end{equation*}
which in turn is guaranteed by our assumptions \eqref{eq:ssb_ups_mu4}, \eqref{eq:ssb_ups_r0} and \eqref{eq:ssb_ups_Gamma0}.

2) Case $\rpol<\rpol_{0}$. Again, we use here the notation introduced in Section~\ref{s:sss} as far as possible. The constant $\nu_{0}$ is chosen as in \eqref{eq:sss_ups_c1}, so that $\ups$ is of class $C^{1}$ even at $\rpol=\rpol_{0}$. 
In order for $\ups$ to be a subsolution, we have to prove ($L_{1}$ is defined in \eqref{eq:sss_ups_diffs})
\begin{equation}
  \label{eq:ssb_ups_ineqs}
  \pder{\ups}{t}
  \le
  \pder{L_{1}}{\rpol}
  +
  L_{1}
  \Big(
  \frac{N-1}{\rpol}
  +
  \exw'(\rpol)
  \Big)
  \,,
  \qquad
  \rpol<\rpol_{0}
  \,.
\end{equation}
Reasoning as in \eqref{eq:sss_ups_diffs2}--\eqref{eq:sss_ups_ineqs2}
we obtain the equivalent version of \eqref{eq:ssb_ups_ineqs}
\begin{equation}
  \label{eq:ssb_ups_ineqs2}
  \begin{split}
    K_{0}&=
           (p-1)
           (t+t_{0})
           \der{}{t}
           \big(
           \tar(\tau)^{\frac{1}{p-1}}
           \big)
           \le
           \trb(\rpol,\tau)
    \\
         &\quad+
           \frac{C_{*}^{p+m-3}(\nu_{0}p)^{p-1}}{(p+m-3)^{p-2}}
           \frac{1}{\excf(\rpol_{0})}
           \Big[
           \frac{p\nu_{0}}{p+m-3}
           \frac{\rpol^{\frac{p}{p-1}}}{\excf(\rpol_{0})^{\frac{1}{p-1}}}
           -
           \trb(\rpol,\tau)
           \Big]
    \\
         &\quad
           -
           \Big(
           \frac{N-1}{\rpol}
           +
           \exw'(\rpol)
           \Big)
           \frac{C_{*}^{p+m-3}(\nu_{0}p)^{p-1}}{(p+m-3)^{p-2}}
           \frac{\trb(\rpol,\tau)\rpol}{\excf(\rpol_{0})}
           \,.
  \end{split}
\end{equation}
However, owing to \eqref{eq:ssb_ups_k0}, \eqref{eq:ssb_ups_ineqs2} is implied by
\begin{equation}
  \label{eq:ssb_ups_ineqs3}
  \begin{split}
    & \frac{p-\alpha_{1}}{\alpha_{1}}
      \frac{1}{\log(t+t_{0})}
      \tar(\tau)^{\frac{1}{p-1}}
      \le
      \trb(\rpol,\tau)
    \\
    &\quad+
      \frac{C_{*}^{p+m-3}(\nu_{0}p)^{p-1}}{(p+m-3)^{p-2}}
      \frac{1}{\excf(\rpol_{0})}
      \Big[
      \frac{p\nu_{0}}{p+m-3}
      \frac{\rpol^{\frac{p}{p-1}}}{\excf(\rpol_{0})^{\frac{1}{p-1}}}
      -
      \trb(\rpol,\tau)
      \Big]
    \\
    &\quad
      -
      \Big(
      \frac{N-1}{\rpol}
      +
      \exw'(\rpol)
      \Big)
      \frac{C_{*}^{p+m-3}(\nu_{0}p)^{p-1}}{(p+m-3)^{p-2}}
      \frac{\trb(\rpol,\tau)\rpol}{\excf(\rpol_{0})}
      \,.
  \end{split}
\end{equation}
When we neglect the positive terms on the right hand side of \eqref{eq:ssb_ups_ineqs3} (excepting the first one), and take into account that $\exw'(\rpol)\rpol\le \alpha_{2}\exw(\rpol)\le \alpha_{2}^{2}\excf(\rpol_{0})$ for $\rpol<\rpol_{0}$, we see that \eqref{eq:ssb_ups_ineqs3} in turn follows from
\begin{equation}
  \label{eq:ssb_ups_ineqs4}
  \begin{split}
    & \frac{p-\alpha_{1}}{\alpha_{1}}
      \frac{1}{\log(t+t_{0})}
      \tar(\tau)^{\frac{1}{p-1}}
      \le
      \trb(\rpol,\tau)
    \\
    &
      \times
      \Big[
      1
      -
      \frac{C_{*}^{p+m-3}(\nu_{0}p)^{p-1}}{(p+m-3)^{p-2}}
      \frac{1}{\excf(\rpol_{0})}
      (
      N
      +
      \alpha_{2}^{2}
      \excf(\rpol_{0})
      )
      \Big]
      =:
      \trb(\rpol,\tau)
      K_{4}
      \,.
  \end{split}
\end{equation}
According to our choice  of a suitably small $C_{*}$ in \eqref{eq:ssb_ups_C} (where we now exploit the third term in the $\max$ function), we have $K_{4}\ge 1/2$. Then \eqref{eq:ssb_ups_ineqs4} is a consequence of
\begin{equation}
  \label{eq:ssb_ups_ineqsE}
  \tar(\tau)^{\frac{1}{p-1}}
  \Big[
  1
  -
  \varGamma
  \frac{p-\alpha_{1}}{\alpha_{1}}
  \frac{2}{\tau}
  \Big]
  \ge
  \rat(\rpol_{0})^{\frac{1}{p-1}}
  =
  \ram(\rpol_{0})
  \ge
  \ram(\rpol)
  \,,
  \quad
  \rpol<\rpol_{0}
  \,.
\end{equation}
Next we remark that
\begin{equation}
  \label{eq:sss_ubs_t00}
  1
  -
  \varGamma
  \frac{p-\alpha_{1}}{\alpha_{1}}
  \frac{2}{\tau}
  \ge
  1
  -
  \frac{p-\alpha_{1}}{\alpha_{1}}
  \frac{2}{\log t_{0}}
  \ge
  \frac{1}{2}
  \,.
\end{equation}
The first inequality in \eqref{eq:sss_ubs_t00} follows immediately from the definition of $\tau$, and the second one from our assumption \eqref{eq:sss_ubs_t0}.
Then, \eqref{eq:ssb_ups_ineqsE} follows for all $\tau>0$ if
\begin{equation}
  \label{eq:ssb_ups_ineqspsib}
  \tar(\varGamma \log t_{0})
  \ge
  2^{p-1}
  \rat(\rpol_{0})
  \,.
\end{equation}
But let us compute from assumption \eqref{eq:ssb_ups_Gamma0} and from the analogue inequalities to \eqref{eq:powerlike_gt}, valid for $\excf$ and $\rat$,
\begin{equation*}
  \begin{split}
    \tar(\varGamma \log t_{0})
    &=
      \tar\big(
      2^{\frac{\alpha_{2}(p-1)}{p-\alpha_{2}}}
      \excf(\rpol_{0})
      \big)
      \ge
      \tar\big(
      \excf\big(2^{\frac{p-1}{p-\alpha_{2}}}\rpol_{0}\big)
      \big)
      \\
    &=
      \rat\big(2^{\frac{p-1}{p-\alpha_{2}}}\rpol_{0}\big)
      \ge
      2^{p-1}
      \rat(\rpol_{0})
      \,.
  \end{split}
\end{equation*}
Hence \eqref{eq:ssb_ups_ineqspsib} is proved, i.e., $\ups$ is a subsolution.

3) Let us estimate the support of $\ups(0)$; owing to \eqref{eq:ssb_ups_ineqspsib}, we have
\begin{equation*}
  \supp \ups(0)
  =
  \{\abs{x}\le \rpol_{1}\}
  \,,
\end{equation*}
for some $\rpol_{1}>\rpol_{0}$; specifically $\rpol_{1}$ is defined by $\rat(\rpol_{1})=\tar(\varGamma \log t_{0})$. On appealing to \eqref{eq:ssb_ups_Gamma0}, this amounts to
\begin{equation*}
  \excf(\rpol_{1})
  =
  \varGamma \log t_{0}
  =
  2^{\frac{\alpha_{2}(p-1)}{p-\alpha_{2}}}
  \excf(\rpol_{0})
  \le
  \excf\big(2^{\frac{\alpha_{2}(p-1)}{\alpha_{1}(p-\alpha_{2})}}\rpol_{0}\big)
  \,.
\end{equation*}
We used here again the inequalities \eqref{eq:powerlike_gt} written with $\exw$ formally replaced by $\excf$. We conclude
\begin{equation}
  \label{eq:ssb_ups_r1}
  \rpol_{1}
  \le
  2^{\frac{\alpha_{2}(p-1)}{\alpha_{1}(p-\alpha_{2})}}\rpol_{0}
  \,.
\end{equation}

4) As for supersolutions, the regularity required in Definition~\ref{d:weaksol} follows from direct inspection and from the explicit form of $\ups_{t}$ found above.

Our next result is essential to the proof of Theorem~\ref{t:sup_sup}.

\begin{lemma}
  \label{l:ssb_comp}
  Let $\ell>0$, $\eps>0$ be given positive numbers. Then it is possible to select $\rpol_{0}$, $C_{*}$, $t_{0}$ and $\varGamma$ so that $\ups$ is a subsolution and
  \begin{gather}
    \label{eq:ssb_comp_n}
    \supp \ups(0)
    \subset
    \{\abs{x}\le \ell\}
    \,,
    \\
    \label{eq:ssb_comp_nn}
    \ups(x,0)
    \le
    \eps
    \,,
    \quad
    x\in\RN
    \,.
  \end{gather}
\end{lemma}

\begin{proof}
  We choose $t_{0}$ as in \eqref{eq:sss_ubs_t0}. Then we select a suitable $\lambda>0$ in \eqref{eq:ssb_ups_r0}, \eqref{eq:ssb_ups_C}, obtaining $\rpol_{0}$, $C_{*}$ and, from \eqref{eq:ssb_ups_Gamma0}, also $\varGamma$ such that, by virtue of \eqref{eq:ssb_ups_r1} the radius $\rpol_{1}$ of the support of $\ups(0)$ satisfies
  \begin{equation}
    \label{eq:ssb_comp_i}
    \rpol_{1}
    \le
    2^{\frac{\alpha_{2}(p-1)}{\alpha_{1}(p-\alpha_{2})}}\rpol_{0}
    \le
    2^{\frac{\alpha_{2}(p-1)}{\alpha_{1}(p-\alpha_{2})}}
    \excf^{(-1)}(\lambda^{p+m-3})
    \,.
  \end{equation}
  Moreover, since $t_{0}>1$, and also invoking the choices of $\rpol_{0}$ and $C_{*}$ in \eqref{eq:ssb_ups_r0} and \eqref{eq:ssb_ups_C}, 
  \begin{equation}
    \label{eq:ssb_comp_ii}
    \begin{split}
      \max_{x\in\RN}
      \ups(x,0)
      &=
        \ups(0,0)
        =
        \frac{C_{*}}{t_{0}^{\frac{1}{p+m-3}}}
        \Big[
        \tar(\varGamma \log t_{0})^{\frac{1}{p-1}}
        -
        (1-\nu_{0})
        \rat(\rpol_{0})^{\frac{1}{p-1}}
        \Big]^{\frac{p-1}{p+m-3}}_{+}
      \\
      &\le
        C_{*}
        \tar(\varGamma \log t_{0})^{p+m-3}
        =
        C_{*}
        \tar\big(2^{\frac{\alpha_{2}(p-1)}{p-\alpha_{2}}}\excf(\rpol_{0})\big)^{p+m-3}
      \\
      &\le
        \lambda\tar\big(2^{\frac{\alpha_{2}(p-1)}{p-\alpha_{2}}}\lambda^{p+m-3}\big)^{p+m-3}
        \,.
    \end{split}
  \end{equation}
  Clearly, by selecting a suitable $\lambda>0$, both \eqref{eq:ssb_comp_n} and \eqref{eq:ssb_comp_nn} are satisfied.
\end{proof}

\section{Transformation to an inhomogeneous density equation}
\label{s:rad}

In this Section we prove Theorem~\ref{t:transform}.
\\
Essentially, we want to move the weight out of the right hand side of \eqref{eq:pde_rad}.
To this end, we perform a change of the independent variable $\rpol$, introducing the function $\rpolf(s)$, $s>0$, and the new unknown $v(s,t)=U(\rpolf(s),t)$. Then clearly
\begin{equation*}
  U_{r}
  =
  v_{s}
  \rpolf_{s}^{-1}
  \,,
  \quad
  U^{m-1}
  \abs{U_{\rpol}}^{p-2}
  U_{\rpol}
  =
  v^{m-1}
  \abs{v_{s}}^{p-2}
  v_{s}
  \abs{\rpolf_{s}}^{-(p-2)}
  \rpolf_{s}^{-1}
  \,.
\end{equation*}
On recalling also
$\partial/\partial\rpol=\rpolf_{s}^{-1}\partial/\partial s$, and on
assuming $\rpolf_{s}>0$, we get that the right hand side of
\eqref{eq:pde_rad} equals
\begin{equation}
  \label{eq:div_new}
  \pder{}{s}
  \big[
  s^{N-1}
  v^{m-1}
  \abs{v_{s}}^{p-2}
  v_{s}
  Y(s)
  \big]
  \rpolf_{s}^{-1}
  \,,
\end{equation}
where
\begin{equation}
  \label{eq:div_coeff}
  Y(s)
  :=
  \ew(\rpolf(s))
  \frac{\rpolf(s)^{N-1}}{s^{N-1}}
  \rpolf_{s}(s)^{-(p-2)-1}
  \,,
  \qquad
  s>0
  \,.
\end{equation}
We choose here $\rpolf$ as a solution to $Y(s)=1$, that is more specifically to the problem
\begin{equation}
  \label{eq:rad_rpolf}
  \rpolf_{s}(s)
  =
  \Big(
  \ew(\rpolf(s))
  \frac{\rpolf(s)^{N-1}}{s^{N-1}}
  \Big)^{\frac{1}{p-1}}
  \,,
  \qquad
  \rpolf(1)
  =
  \rpol_{*}
  \,.
\end{equation}
We prove in Section~\ref{s:app} that for a suitable $\rpol_{*}>0$, depending only on $N$, $p$ and $\exw$, the solution $\rpolf(s)$ is defined and increasing over $(0,+\infty)$ and satisfies
\begin{equation*}
  \lim_{s\to 0+}
  \rpolf(s)
  =
  0
  \,,
  \qquad
  \lim_{s\to +\infty}
  \rpolf(s)
  =
  +\infty
  \,,
\end{equation*}
so that it can be extended to a continuous function over $[0,+\infty)$. Even more, we have $\rpolf\in C^{1}([0,+\infty))$, with $\rpolf_{s}(0)=1$. 
\\
Then \eqref{eq:pde_rad} becomes
\begin{equation*}
  \rpolf^{N-1}
  \ew(\rpolf)
  \rpolf_{s}
  \pder{v}{t}
  =
  \pder{}{s}
  \big[
  s^{N-1}
  v^{m-1}
  \abs{v_{s}}^{p-2}
  v_{s}
  \big]
  \,,
\end{equation*}
that is in view of our choice $Y(s)=1$,
\begin{equation}
  \label{eq:rad_new}
  \rpolf_{s}^{p}
  \pder{v}{t}
  =
  \frac{1}{s^{N-1}}
  \pder{}{s}
  \big[
  s^{N-1}
  v^{m-1}
  \abs{v_{s}}^{p-2}
  v_{s}
  \big]
  \,.
\end{equation}
The right hand side of \eqref{eq:rad_new} can of course be understood as the divergence in the space variable $x$ such that $\abs{x}=s$; we use for the sake of simplicity the old variable names. Then if we define
\begin{equation}
  \label{eq:rad_rw}
  \rw(s)=\rpolf_{s}(s)^{p}
  \,,
\end{equation}
we have the equation \eqref{eq:rad_N}
for the radial function $v$. According to the results on $\rpolf$ outlined above, $\rw$ is a continuos positive function on $[0,+\infty)$.

Next we estimate the asymptotic behavior of $\rw(s)$. We denote by $C_{i}$ constants depending only on $N$, $p$, $\exw$, possibly varying from line to line. We get from \eqref{eq:rad_rpolf} that
\begin{equation*}
  \frac{1}{(\ew(\rpol)\rpol^{N-1})^{\frac{1}{p-1}}}
  =
  \frac{1}{\flopr(\rpol)^{\frac{N-1}{p-1}}}\flopr_{\rpol}(\rpol)
  \,,
\end{equation*}
where $\flopr=\rpolf^{(-1)}$.
Thus, on defining
\begin{equation*}
  \varphi(\rpol)
  :=
  \int_{\rpol}^{+\infty}
  \frac{\di z}{(\ew(z)z^{N-1})^{\frac{1}{p-1}}}
  \,,  
\end{equation*}
we get
\begin{equation}
  \label{eq:rad_rs}
  \varphi(\rpol)
  =
  \int_{\rpol}^{+\infty}
  \frac{1}{\flopr(z)^{\frac{N-1}{p-1}}}\flopr_{z}(z)
  \di z
  =
  \frac{p-1}{N-p}
  \flopr(\rpol)^{-\frac{N-p}{p-1}}
  \,.
\end{equation}
Define the function
\begin{equation}
  \label{eq:rad_H}
  \Hf(r)
  =
  \Big[
  \frac{\exw(\rpol)}{\rpol}
  \big(
  \ew(\rpol)
  \rpol^{N-1}
  \big)^{\frac{1}{p-1}}
  \Big]^{-1}
  \,,
  \qquad
  \rpol>0
  \,.
\end{equation}
Let us calculate, recalling that $\ew(\rpol)=e^{\exw(\rpol)}$,
\begin{equation}
  \label{eq:rad_H_der}
  \begin{split}
    \Hf'(\rpol)
    &=
      -
      \Big[
      \frac{\exw(\rpol)}{\rpol}
      \big(
      \ew(\rpol)
      \rpol^{N-1}
      \big)^{\frac{1}{p-1}}
      \Big]^{-2}
      \big(
      \ew(\rpol)
      \rpol^{N-1}
      \big)^{\frac{1}{p-1}}
      \\
    &\quad
      \times\Big\{
      \Big[
      \frac{\exw'(\rpol)}{\rpol}
      -
      \frac{\exw(\rpol)}{\rpol^{2}}
      \Big]
      +
      \frac{\exw(\rpol)}{(p-1)\rpol}
      \big(
      \ew(\rpol)
      \rpol^{N-1}
      \big)^{-1}
  \\
    &\qquad
      \times\big[
      \exw'(\rpol)
      \ew(\rpol)
      \rpol^{N-1}
      +
      (N-1)
      \ew(\rpol)
      \rpol^{N-2}
      \big]
      \Big\}
      \\
    &=
      -
      \Big[
      \frac{\exw(\rpol)}{\rpol}
      \big(
      \ew(\rpol)
      \rpol^{N-1}
      \big)^{\frac{1}{p-1}}
      \Big]^{-1}
      \Big\{
      \Big[
      \frac{\exw'(\rpol)}{\exw(\rpol)}
      -
      \frac{1}{\rpol}
      \Big]
      \\
    &\quad
      +
      \frac{1}{p-1}
      \Big[
      \exw'(\rpol)
      +
      \frac{N-1}{\rpol}
      \Big]
      \Big\}
      \,.
  \end{split}
\end{equation}
Thus
\begin{equation*}
  \begin{split}
    \frac{\Hf'(\rpol)}{\varphi'(\rpol)}
    &=
      \Big[
      \frac{\exw(\rpol)}{\rpol}
      \Big]^{-1}
      \Big\{
      \Big[
      \frac{\exw'(\rpol)}{\exw(\rpol)}
      -
      \frac{1}{\rpol}
      \Big]
      \\
    &\quad
      +
      \frac{1}{p-1}
      \Big[
      \exw'(\rpol)
      +
      \frac{N-1}{\rpol}
      \Big]
      \Big\}
      \\
    &=
      \frac{1}{\exw(\rpol)}
      \Big[
      \frac{\exw'(\rpol)\rpol}{\exw(\rpol)}
      -
      1
      \Big]
      +
      \frac{1}{p-1}
      \Big[
      \frac{\exw'(\rpol)\rpol}{\exw(\rpol)}
      +
      \frac{N - 1}{\exw(\rpol)}
      \Big]
      \,.
  \end{split}
\end{equation*}
Then, by virtue of  \eqref{eq:powerlike}, for all $\rpol>0$ we infer
\begin{equation}
  \label{eq:rad_der_ratio}
  \begin{split}
    \frac{\alpha_{1}}{p-1}
    &\le
      \Big[
      \alpha_{1}-1
      +
      \frac{N-1}{p-1}
      \Big]
      \frac{1}{\exw(\rpol)}
      +
      \frac{\alpha_{1}}{p-1}
      \le
      \frac{\Hf'(\rpol)}{\varphi'(\rpol)}
    \\
    &\le
      \Big[
      \alpha_{2}-1
      +
      \frac{N-1}{p-1}
      \Big]
      \frac{1}{\exw(\rpol)}
      +
      \frac{\alpha_{2}}{p-1}
      \,,
  \end{split}
\end{equation}
whence
\begin{equation}
  \label{eq:rad_der_ratio2}
  -
  C_{1}
  \varphi'(\rpol)
  \le
  -
  \Hf'(\rpol)
  \le
  -
  C_{2}
  \varphi'(\rpol)
  \,,
  \qquad
  \rpol \ge \rpol_{*}
  \,.
\end{equation}
Thus, by integration over $(\rpol,+\infty)$ we get
\begin{equation}
  \label{eq:rad_Hf_equiv}
  C_{1}
  \varphi(\rpol)
  \le
  \Hf(\rpol)
  \le
  C_{2}
  \varphi(\rpol)
  \,,
  \qquad
  \rpol \ge \rpol_{*}
  \,.
\end{equation}
Then, by collecting \eqref{eq:rad_rs}, \eqref{eq:rad_Hf_equiv}, we find
\begin{equation}
  \label{eq:rad_Hf_equiv2}
  C_{1}
  s^{-\frac{N-p}{p-1}}
  \le
  \Hf(\rpolf(s))
  \le
  C_{2}
  s^{-\frac{N-p}{p-1}}
  \,,
  \qquad
  s\ge 1
  \,,
\end{equation}
that is, by taking the logarithm and invoking \eqref{eq:powerlike_gt}--\eqref{eq:powerlike_lt},
\begin{equation}
  \label{eq:rad_r_asy}
  \rpolf(s)
  \sim
  \exw^{(-1)}(\log s)
  \,.
\end{equation}
\\
Then, from the differential equation \eqref{eq:rad_rpolf}, the definition \eqref{eq:rad_H} of $\Hf$ and the estimate \eqref{eq:rad_Hf_equiv2}, we get (understanding $\rpolf=\rpolf(s)$)
\begin{equation}
  \label{eq:rad_rs_est}
  \begin{split}
    \rpolf_{s}(s)
    &=
      \big(
      \ew(\rpolf)
      \rpolf^{N-1}
      \big)^{\frac{1}{p-1}}
      s^{-\frac{N-1}{p-1}}
      =
      \Hf(\rpolf)^{-1}
      \frac{\rpolf}{\exw(\rpolf)}
      s^{-\frac{N-1}{p-1}}
      \\
    &\sim
      \frac{\exw^{(-1)}(\log s)}{\log s}
      \frac{1}{s}
      \,.
  \end{split}
\end{equation}
Our estimate \eqref{eq:rad_rw_est} follows at once.

\subsection{Details on the function $\rpolf$}
\label{s:app}

We rewrite \eqref{eq:rad_rpolf} as
\begin{equation}
  \label{eq:app_ode}
  \rpolf'(s)
  =
  \nonl(\rpolf)
  s^{-1-\beta}
  \,,
  \qquad
  \rpolf(1)
  =
  \rpol_{0}
  \,,
\end{equation}
with
\begin{equation*}
  \nonl(\rpolf)
  =
  (\ew(\rpolf)
  \rpolf^{N-1}
  )^{\frac{1}{p-1}}
  \,,
  \qquad
  \beta
  =
  \frac{N-p}{p-1}
  >0
  \,.
\end{equation*}
Consider the solution $\rpolfb(z)$ to
\begin{equation}
  \label{eq:app_ode_b}
  \rpolfb'
  =
  \nonl(\rpolfb)
  \,,
  \qquad
  \rpolfb(0)
  =
  \rpol_{*}
  \,,
\end{equation}
where $\rpol_{*}$ will be chosen presently.
Owing to the superlinearity and positivity of $\nonl$, and also taking into account that $\nonl(0)=0$, the solution $\rpolfb$ is positive over $(-\infty,z_{1})$, blows up at some $z_{1}(\rpol_{*})\in (0,+\infty)$ and satisfies $\rpolfb(z)\to 0$ as $z\to-\infty$. The blow up point $z_{1}$ depends monotonically and continuosly on $\rpol_{*}$, ranging over $(0,+\infty)$. We select $\rpol_{*}$ so that $z_{1}=1/\beta$; note that $\rpol_{*}$ depends only on $N$, $p$ and $\exw$.
\\
Then the solution to \eqref{eq:app_ode} can be obtained as
\begin{equation*}
  \rpolf(s)
  :=
  \rpolfb\Big(
  \frac{1-s^{-\beta}}{\beta}
  \Big)
  \,,
\end{equation*}
which is defined in the interval
\begin{equation*}
  -\infty
  <
  \frac{1-s^{-\beta}}{\beta}
  <
  z_{1}
  =
  \frac{1}{\beta}
  \,,
  \quad
  \text{i.e.,}
  \quad
  0<s<+\infty
  \,.
\end{equation*}
Next we rewrite \eqref{eq:rad_rpolf} as
\begin{equation}
  \label{eq:app_ode_c}
  \rpolf_{s}
  =
  \nonlb(s)
  \rpolf^{1+\beta}
  \,,
  \qquad
  \rpolf(1)
  =
  \rpol_{*}
  \,.
\end{equation}
with $\beta$ as above and
\begin{equation*}
  \nonlb(s)
  =
  \ew(\rpolf(s))^{\frac{1}{p-1}}
  s^{-1-\beta}
  \,,
  \qquad
  s>0
  \,.
\end{equation*}
From \eqref{eq:app_ode_c} we get by elementary integration by separation of variables,
\begin{equation}
  \label{eq:app_sepv}
  \rpolf(s)
  =
  \Big[
  \rpol_{*}^{-\beta}
  -
  \beta
  \int_{1}^{s}
  \nonlb(y)
  \di y
  \Big]^{-\frac{1}{\beta}}
  \,,
  \qquad
  s>0
  \,,
\end{equation}
so that by a direct calculation
\begin{equation}
  \label{eq:app_der}
  \rpolf_{s}(s)
  =
  \Big[
  \nonlb(s)^{-\frac{\beta}{\beta+1}}
  \rpol_{*}^{-\beta}
  -
  \beta
  \nonlb(s)^{-\frac{\beta}{\beta+1}}
  \int_{1}^{s}
  \nonlb(y)
  \di y
  \Big]^{-\frac{\beta+1}{\beta}}
  \,.
\end{equation}
Note that according to the analysis above, $\ew(\rpolf(s))\to 1$ as $s\to 0+$, so that $\nonlb(s)^{-\beta/(\beta+1)}\to 0+$ as $s\to 0+$, while
\begin{equation*}
  \int_{1}^{s}
  \nonlb(y)
  \di y
  \to
  -\infty
  \,,
  \qquad
  s\to 0+
  \,.
\end{equation*}
Then, by L'H\^{o}pital's rule, we get from \eqref{eq:app_der} and the definition of $\nonlb$
\begin{equation*}
  \begin{split}
    \rpolf_{s}(0)^{-\frac{\beta}{\beta+1}}
    &=
      -
      \beta
      \lim_{s\to 0+}
      \ew(\rpolf(s))^{-\frac{\beta}{(\beta+1)(p-1)}}
      s^{\beta}
      \int_{1}^{s}
      \nonlb(y)
      \di y
      \\
    &=
      -
      \beta
      \lim_{s\to 0+}
      s^{\beta}
      \int_{1}^{s}
      \ew(\rpolf(y))^{\frac{1}{p-1}}
      y^{-1-\beta}
      \di y
     \\
    &=
      -
      \beta
      \lim_{s\to 0+}
      \frac{\ew(\rpolf(s))^{\frac{1}{p-1}}
      s^{-1-\beta}}{-\beta s^{-1-\beta}}
      =
      1
      \,.
  \end{split}
\end{equation*}

\def\cprime{$'$}


\begin{thebibliography}{10}

\bibitem{Andreucci:Tedeev:2015}
D.~Andreucci and A.~F. Tedeev.
\newblock Optimal decay rate for degenerate parabolic equations on noncompact
  manifolds.
\newblock {\em Methods Appl. Anal.}, 22(4):359--376, 2015.

\bibitem{Andreucci:Tedeev:2021b}
D.~Andreucci and A.~F. Tedeev.
\newblock Asymptotic properties of solutions to the {C}auchy problem for
  degenerate parabolic equations with inhomogeneous density on manifolds.
\newblock {\em Milan J. Math.}, 89(2):295--327, 2021.

\bibitem{Andreucci:Tedeev:2022b}
D.~Andreucci and A.~F. Tedeev.
\newblock Existence of solutions of degenerate parabolic equations with
  inhomogeneous density and growing data on manifolds.
\newblock {\em Nonlinear Anal.}, 219:Paper No. 112818, 15, 2022.

\bibitem{Andreucci:Tedeev:2025}
D.~Andreucci and A.~F. Tedeev.
\newblock The {C}auchy problem for doubly degenerate parabolic equations with
  weights.
\newblock {\em NoDEA Nonlinear Differential Equations Appl.}, 32(2):Paper No.
  26, 2025.

\bibitem{Degtyarev:Tedeev:2012}
S.~P. Degtyarev and A.~F. Tedeev.
\newblock On the solvability of the {C}auchy problem with growing initial data
  for a class of anisotropic parabolic equations.
\newblock {\em Ukr. Mat. Visn.}, 8(3):356--380, 461, 2011.

\bibitem{Grigoryan:2006}
A.~Grigor{\cprime}yan.
\newblock Heat kernels on weighted manifolds and applications.
\newblock In {\em The ubiquitous heat kernel}, volume 398 of {\em Contemp.
  Math.}, pages 93--191. Amer. Math. Soc., Providence, RI, 2006.

\bibitem{Grillo:Muratori:2016}
G.~Grillo and M.~Muratori.
\newblock Smoothing effects for the porous medium equation on
  {C}artan-{H}adamard manifolds.
\newblock {\em Nonlinear Anal.}, 131:346--362, 2016.

\bibitem{Grillo:Muratori:Vazquez:2017}
G.~Grillo, M.~Muratori, and J.~L. V\'azquez.
\newblock The porous medium equation on {R}iemannian manifolds with negative
  curvature. {T}he large-time behaviour.
\newblock {\em Adv. Math.}, 314:328--377, 2017.

\bibitem{Muratori:2021}
M.~Muratori.
\newblock Some recent advances in nonlinear diffusion on negatively-curved
  {R}iemannian manifolds: from barriers to smoothing effects.
\newblock {\em Boll. Unione Mat. Ital.}, 14(1):69--97, 2021.

\bibitem{Muratori:Roncoroni:2022}
M.~Muratori and A.~Roncoroni.
\newblock Sobolev-type inequalities on {C}artan-{H}adamard manifolds and
  applications to some nonlinear diffusion equations.
\newblock {\em Potential Anal.}, 57(1):129--154, 2022.

\bibitem{Tedeev:2007}
A.~F. Tedeev.
\newblock The interface blow-up phenomenon and local estimates for doubly
  degenerate parabolic equations.
\newblock {\em Appl. Anal.}, 86(6):755--782, 2007.

\bibitem{Tedeev:2025}
A.~F. Tedeev.
\newblock Some qualitative properties of solutions to the {C}auchy problem of
  the degenerate parabolic equations with a drift term.
\newblock {\em Differential Integral Equations}, 38(11-12):745--774, 2025.

\bibitem{Tsutsumi:1988}
M.~Tsutsumi.
\newblock On solutions of some doubly nonlinear parabolic equations with
  absorption.
\newblock {\em Journal of Mathematical Analysis and Applications},
  132:187--212, 1988.

\bibitem{Vazquez:2015}
J.~L. V\'{a}zquez.
\newblock Fundamental solution and long time behavior of the porous medium
  equation in hyperbolic space.
\newblock {\em J. Math. Pures Appl. (9)}, 104(3):454--484, 2015.

\end{thebibliography}
\end{document}